%% file: Jiggling+volume-revision-JSG.tex
\documentclass[a4paper, 11pt]{article}

\usepackage[T1]{fontenc}
\usepackage{amsfonts,amssymb}
\usepackage{amsmath}
\usepackage{enumerate}
\usepackage{amsthm}
\usepackage{color}
\usepackage{url}
\usepackage{appendix}
\usepackage[utf8]{inputenc}
\usepackage{geometry}
\usepackage{pgfplots}
\usepackage{graphicx}
\usepackage{color}
\usepackage{transparent}
\usepackage{setspace}
\usepackage{pdfsync}

\renewcommand{\a}{{\mathcal A}{}}

\newcommand{\D}{{\mathcal D}{}}

\newcommand{\N}{{\mathbb N}{}}

\newcommand{\R}{\mathbb R}
\newcommand{\ES}{{\mathcal S}{}}

\newcommand{\ol}{\overline}
\newcommand{\ul}{\underline}

\newcommand{\st}{\mathrm{st}}

\renewcommand{\ES}{{\mathcal S}{}}

\newcommand{\rp}{respectively }
\newcommand{\om}{\omega}
\newcommand{\Om}{\Omega}

\newtheorem{thm}{Theorem}[section]
\newtheorem{cor}[thm]{Corollary}
\newtheorem{prop}[thm]{Proposition}
\newtheorem{lemma}[thm]{Lemma}
\theoremstyle{remark}
\newtheorem{rem}[thm]{Remark}

\theoremstyle{definition}
\newtheorem{df}[thm]{Definition}

\def\sref#1{Section~\ref{#1}}

\def\tref#1{Theorem~\ref{#1}}

\def\rref#1{Remark~\ref{#1}}
\def\cref#1{Corollary~\ref{#1}}
\def\pref#1{Proposition~\ref{#1}}
\def\lref#1{Lemma~\ref{#1}}

\newcommand{\intro}[1]
{\renewcommand{\thesection}{\fnsymbol{section}}
\setcounter{section}{-1}
\section{#1}
\renewcommand{\thesection}{\arabic{section}}
}

\title{PL approximations of symplectic manifolds}
\author{M\'elanie Bertelson \& Julie Distexhe}

\begin{document}

\maketitle

\tableofcontents

\intro{Introduction}

This paper stems from Eliashberg's idea to develop a theory of piecewise linear symplectic topology \cite{YE}. The central object of this theory is a \emph{piecewise linear symplectic manifold}, that is a piecewise linear (PL) manifold $P$, of even dimension $2n$, endowed with a piecewise constant symplectic structure $\om$. The latter is a collection $(\om_\sigma)_{\sigma}$ of constant symplectic forms on the various top-dimensional simplices of a triangulation of $P$ such that for any two adjacent $2n$-simplices $\sigma$ and $\sigma'$, the restriction of $\om_\sigma$ to $\sigma \cap \sigma'$ agree with that of $\omega_{\sigma'}$. Forms satisfying the latter condition are called Whitney forms in \cite{sullivan}. (Definitions pertaining to PL topology are to be found in \sref{Prelims}.)

\smallskip

One of the very first questions is how these compare to \ul{smooth} symplectic manifolds. For instance, is any smooth symplectic manifold $(M, \omega)$ piecewise differentiably equivalent to a PL one~? In other words, does there exists a PL manifold $P$ and a piecewise differentiable maps $h : P \to M$ for which $h^*\om$ is piecewise constant~? If yes, we say that $(M, \om)$ can be triangulated. 

\smallskip

This question is closely related to another very first question which is to decide whether a PL symplectic manifold $(W, \om)$ satisfies Darboux, in the sense that for any point $w$ of $W$, there exists a PL homeomorphism $h$ between some  neighborhood of $w$ in $W$ and an open subset of $\R^{2n}$, through which $\om$ corresponds to the standard symplectic form $\om_o = \sum_{i=1}^ndx^i \wedge dy^i$ on $\R^{2n}$. Indeed, a PL symplectic manifold which triangulates a smooth one necessarily satisfies Darboux (see \cite{D19}).

\smallskip

Both questions are surprisingly difficult to settle. It is not unlikely that Darboux fails in general. There might indeed exist some type of curvature for PL symplectic forms. Triangulations, on the other hand, exist for $2$-dimensional manifolds or products of such~: we indeed prove, in the first part of the paper, that smooth \ul{volume} forms can be uniquely triangulated. We also prove, in the second part of the paper, that smooth symplectic manifold can be approximated, in the $C^1$ sense, by PL ones. One can moreover ensure that the piecewise constant volume associated to some of those approximating PL symplectic forms triangulate the smooth symplectic volume.

\smallskip

The literature on PL symplectic topology is quite scarce. We would like to mention Bernhard Gratza's unpublished thesis \cite{gratza}, which shows that for any compactly supported Hamiltonian diffeomorphism $\varphi$ of $(\R^{2n}, \om_o)$, there exists two triangulations $K_1$ and $K_2$ of $\R^{2n}$ and a piecewise linear symplectic diffeomorphism $\psi : K_1 \to K_2$ that $C^0$-approximates $\varphi$ (of course, this result does not permit to replace a smooth atlas by a PL one). Panov in \cite{panov} studied polyhedral K\"ahler manifold but from a different viewpoint. More recently, Jauberteau, Rollin and Tapie have proven in \cite{rollin} that an isotropic immersed torus $\mathbb T^k$ in $(\R^{2n}, \om_o)$ can be approximated by isotropic piecewise linear maps $\mathbb T^k \to \R^{2n}$. Rollin has proven a refinement of that result in \cite{rollin2020polyhedral}. The same author has also developed an approach to construct PL symplectic homeomorphisms from the $4$-torus to itself in \cite{rollin2021polyhedral}. \\


In the first part of the paper we prove the following result about triangulation of smooth volume forms. 

\begin{thm}\label{Theorem} Given a volume form $\Omega$ on a manifold $M$ and a Whitehead triangulation $h : |K| \to M$, there exists a refinement $\overline{K}$ of $K$ and a piecewise differentiable map $\ol{h} : |\overline{K}| \to M$ such that $\ol{h}^*\Omega$ is constant on each simplex of $\overline{K}$. Moreover, if $g : |K| \to M$ and $h : |L| \to M$ are two triangulations of the same volume form, there exists a piecewise linear homeomorphism $\psi : |K| \to |L|$ such that $g = h \circ \psi$. 
\end{thm}

Observe, first of all, that, given a smooth triangulation $h : |K| \to M$ of $M$, there is a natural candidate for the role of the triangulating volume form, namely, the unique piecewise constant $(\dim M)$-form on $K$ attributing to each simplex of $K$ the same orientation and volume as $h^*\Omega$. Denote it by $\Omega_1$ and denote the form $h^*\Omega$ by $\Omega_0$. If there exists a piecewise smooth homeomorphism $\varphi : \ol{K} \to \ol{K}$, defined on a refinement $\ol{K}$ of $K$, that pulls back $\Omega_0$ to $\Omega_1$, then $\ol{h} = h \circ \varphi :  \ol{K} \to M$ triangulates~$\Omega$.
 
\smallskip

Now consider the vector field defined by the equation $\iota(X_t) \Omega_t = \alpha$, where $\Omega_t = \Omega_0 + t (\Omega_1 - \Omega_0)$ and $\alpha$ is a Whitney form such that $d\alpha = \Omega_0 - \Omega_1$. It is multi-valued along the skeleton of $K$ and its flow is therefore not obviously well-defined. It might therefore be wise to explore another avenue. The one adopted here is to arrange for the form $\alpha$ to vanish along the skeleton of some subdivision $\ol{K}$ of $K$, so that the flow of $X_t$ is stationary there. It preserves therefore any simplex and is thus well-defined on $\ol{K}$. 

\smallskip

Observe that a piecewise differentiable homeomorphism $\varphi : |K| \to |K|$ that restricts to the identity along the skeleton of $K$ satisfies $\varphi_{*_x} = \rm Id$ at all points $x$ in the codimension $2$ skeleton. So that in order to be able to find such a primitive $\alpha$, it is necessary to modify the initial triangulation in order to achieve that $\Omega_0$ and $\Omega_1$ coincide along the codimension $2$ skeleton, and this is the delicate point of the proof. We end up proving in fact more~: for a suitable choice of $h$, the forms $\Omega_0$ and $\Omega_1$ agree on a neighborhood of the entire skeleton of some subdivision $\ol{K}$ of $K$. 

\smallskip

The next question is that of uniqueness~: given a closed manifold $M$ endowed with a volume form $\Om$ and two triangulations $g : |K| \to M$ and $h : |L| \to M$, does there necessarily exist a PL equivalence $\psi : |K| \to |L|$ such that $h \circ \psi = g$~? The answer is yes.  In fact any two piecewise constant volume forms on a same PL manifold that have the same total volume are PL equivalent. \\

In the second part of the paper, we prove that any smooth symplectic $(M, \om)$ manifold admits a triangulation $h : |K| \to M$ that is in general position with respect to $\om$, that is for which $\om$ has maximal rank along all the simplices of $K$. In the last part of the paper, we used this condition on $K$ to construct approximations of $\om$ but, as observed by the referee, general position is superfluous. We think the jiggling lemma is nevertheless interesting in itself and deserves to be written here.

\begin{lemma}\label{sympjig}(Symplectic Jiggling Lemma) Let $(M, \omega)$ denote a smooth symplectic manifold. Consider a Whitehead triangulation $h : |K| \to M$ of $M$. Then for a crystalline subdivision $\mathcal{S}_\ell(K)$ of $K$ of sufficiently large order $\ell$, there exist arbitrary small perturbations of $h : \mathcal{S}_\ell(K) \to M$ that are in general position with respect to~$\omega$. 
\end{lemma}

The proof is largely inspired by Thurston's jiggling lemma, which asserts that if $M$ is a smooth manifold equipped with a $C^0$ distribution $\tau$, there exist triangulations of $M$ which are in general position with respect to $\tau$ \cite{thurstonjig}. Despite the intuitive evidence of the later result, the proof is, in our opinion, quite subtle and remarkably elegant. \\

In the third and last part of the paper, we prove that any smooth symplectic manifold $(M, \om)$ can be $C^1$-approximated by PL ones. Indeed, given a triangulation $h : K \to M$, there is a natural procedure to construct a piecewise constant symplectic form ${\mathcal C}_K(\omega)$ on $K$. Moreover, the latter can be made arbitrarily $C^1$-close to $\om$ by subdividing sufficiently finely the initial triangulation. This leads to the following result.

\begin{thm} Let $(M, \om)$ be a closed symplectic manifold, and $h : |K| \to M$, a smooth triangulation. Given $\varepsilon > 0$, there exists an $\ell \geq 0$ such that ${\mathcal C}_{{\mathcal S}^\ell(K)}(\om)$ is $\varepsilon$-close to $\om$ in the $C^0$-norm.
\end{thm}

This provides in particular loads of examples of PL symplectic manifolds and says also that it is sufficient for a PL manifold to admit a smoothing that supports a smooth symplectic form to support a PL one. It might also be that there exists smooth manifolds which admit PL symplectic structures but no smooth ones. This cannot be the case for closed manifolds with no degree $2$ cohomology class whose top exterior power is not exact, as the standard argument to prove that such do not support any symplectic structure is valid in the PL setting as well. (Thus even-dimensional spheres do not support PL symplectic structures, except in dimension $2$).

\smallskip

Coming back to our initial project, one could hope that some of these approximating piecewise constant forms ${\mathcal C}_{{\mathcal S}^\ell(K)}(\om)$ on $|K|$ are piecewise smoothly equivalent to the initial smooth one. Unfortunately this might not be so because it is not at all clear that the form we construct do satisfy Darboux. It might be necessary to invent another procedure to produce PL symplectic forms that would enforce Darboux from the start. 

\smallskip

Let us mention also that Moser's argument used to construct symplectomorphisms in various situations does not easily adapt to our setting. Indeed, as explained above in a slightly different context, the first part of the argument goes true and a vector field can be constructed, but it is multivalued along the skeleton. As a result, it is in general impossible to make sense of its flow as a family of piecewise differentiable homeomorphisms. One obvious obstruction is that a flow line may reach the skeleton in finite time, where bifurcation can occur. On the other hand we have tried to deform the initial triangulation $h : |K| \to M$ so as to make $h^*\om$ coincide with ${\mathcal C}_K(h^*\omega)$ near the skeleton, so that we can produce a Moser vector field that vanish near the skeleton, as is done for volume forms. For volume forms though, the piecewise constant form constructed initially never changes throughout the process since the deformation we construct is a reparameterization of the initial triangulation. For symplectic forms, because the initial piecewise constant form might not satisfy Darboux, we may not hope to achieve our goal by simply reparameterizing the initial triangulation and we might be forced to change it radically. We did not find a concrete way to proceed. \\

The paper is organized as follows. The first section presents the definitions and notations adopted in the text. The second one contains the proof of the existence part of \tref{Theorem}. It is itself divided into a first subsection that explains the induction procedure to construct a triangulation making $\Omega_0$ agrees with $\Omega_1$ near the skeleton and a second one that contains a technical lemma. The uniqueness part is proven in the third section and consequences of \tref{Theorem} appear in the Fourth section. The fifth section is devoted to the symplectic jiggling lemma. It is divided into three subsections. The first one discusses the statement of the theorem. It is followed by a reminder of the notion of crystalline subdivision of a simplicial complex. Then comes the proof itself. Section \ref{construction PC} explains how to PL approximate smooth symplectic structures. \\

{\bf Acknowledgments} We are particularly grateful to Yakov Eliashberg who initiated the project to investigate PL/PS symplectic and contact topology and with whom the first author has had numerous very stimulating and interesting conversations. Part of the research was carried out during a visit of the authors, as well as several visits of the first author at the Mathematics Department of Stanford University that we thank for its hospitality. We also thank Fran\c cois Laudenbach, whom we have consulted on several occasions. The research of the first author was supported by the EoS grant \emph{Symplectic Techniques in Differential Geometry}. The second author was supported by a grant of the FRS-FNRS. Lastly, the authors thank the referee for noticing that the jiggling lemma was not necessary to carry out the approximation procedure.

\section{Definitions and Notations}\label{Prelims}

The conventions, notations and definitions adopted in the text are described hereafter. We begin with those that concern PL topology; they agree with Lurie's in \cite{Lurie}. We end with those specific to symplectic PL topology.

\subsection{PL topology} 

By a $n$-simplex $\sigma$, we mean a closed affine simplex of dimension $n$ in some Euclidean space $\R^N$, or, equivalently, the convex hull of a collection of $n+1$ independent points $v_0, \ldots , v_n \in \R^N$ (\emph{independent} means that none is contained in the affine plane generated by the other ones). We write $\sigma = \langle v_0, \ldots , v_n\rangle$. More generally, given subsets $A_1, ..., A_k$ of $\R^N$, the symbol $\langle A_1, ..., A_k \rangle$ stands for the convex hull of the union of the $A_i$'s. We also denote by $L_\sigma$ the vector space parallel to the affine space spanned by $\sigma$. In the text, a simplex $\sigma$ is sometimes also thought of as the simplicial complex that consists $\sigma$ plus its faces. 

\medskip

We think of a simplicial complex $K$ as being a collection of simplices in some Euclidean space $\R^N$ that satisfies the usual properties~: a face of a simplex in $K$ belongs to $K$ and the intersection of two simplices $\sigma_1$ and $\sigma_2$ in $K$ is either empty or consists of a common face of $\sigma_1$ and $\sigma_2$. If $K$ is not a finite collection then we require, in addition, that $K$ be locally finite. For $n \in \N$, the collection of all closed $n$-simplices of $K$ is denoted by $K^{(n)}$. 

\medskip

The \emph{star} of a simplex $\sigma \in K$, denoted by $\rm st(\sigma)$ or $\rm st_K(\sigma)$, is the subcomplex of $K$ consisting of all simplices in $K$ that intersect $\sigma$ non-trivially, together with their faces. The \emph{link} of $\sigma$, denoted by $\rm lk(\sigma)$ or $\rm lk_K(\sigma)$, consists in the simplices of $\rm st(\sigma)$ that do not intersect $\sigma$. The \emph{body} of a simplicial complexe $L$, denoted by $|L|$, is the union $|L| = \cup_{\sigma \in K} \sigma$. A \emph{polyhedron} is defined to be the body of some simplicial complexe. In turn a simplicial complexe $K$ is a \emph{triangulation} of the polyhedron $|K|$. 

\medskip

A \emph{subdivision} or \emph{refinement} of a simplicial complex $K$ is a simplicial complex $L$ such that each simplex of $L$ is contained in a simplex of $K$ and each simplex of $K$ is a union of simplices of $L$. In particular their polyhedra $|K|$ and $|L|$ are identical. 

\medskip

A \emph{piecewise linear (PL) embedding} of a simplicial complex $K$ into some Euclidean space $\R^M$ is a continuous embedding of its body into $\R^M$ whose restriction to each simplex $\sigma$ is an affine map. Given two simplicial complexes $K \subset \R^M$ and $L \subset \R^N$, a \emph{PL homeomorphism} from $K$ to $L$, denoted $h : K \to L$, is a homeomorphism $h : |K| \to |L|$ whose restriction to each simplex $\sigma \in K$ is an affine map onto a simplex $h(\sigma) \in L$. If $K$ and $L$ admit subdivisions $K'$ and $L'$ and a PL homeomorphism $h : K' \to L'$, then $K$ and $L$ are said to be \emph{PL equivalent}. Similarly a homeomorphism $h : |K| \to |L|$ is said to be \emph{piecewise smooth (PS)} if each simplex of $K$ is smoothly embedded onto a simplex of $L$. A \emph{piecewise smooth equivalence} is defined accordingly. 

\medskip

A \textsl{PL manifold} of dimension $n$ is a polyhedron $P$ which is locally flat, in the sense that for some (in fact, any) triangulation $K$ of $P$, the star of any simplex $\sigma \in K$ admits a PL embedding onto a PL ball in $\R^n$. Two PL manifolds $P$ and $P'$ are said to be PL (\rp PS) equivalent if they admit PL (\rp PS) equivalent triangulations. The collection of PL manifolds with PL equivalences form a category (see \cite{Lurie}). 

\medskip

Given a simplex $\sigma$, its barycenter is denoted by $b$ or $b_\sigma$. The \emph{stellar subdivision} of a $n$-simplex $\sigma$, denoted by $\mathcal{S}\sigma$, is the simplicial complex consisting of all proper faces of $\sigma$, together with the vertex $b_\sigma$ and the simplices $<\tau, b_\sigma>$, where $\tau$ runs through the collection of proper faces of $\sigma$. The stellar subdivision of a PL manifold $K$ of dimension $n$, denoted by $\ES K$, is the union over all $\sigma \in K^{(n)}$ of the $\mathcal{S}\sigma$'s. Thus only the top-dimensional faces are subdivided. 

\medskip

Finally we adopt Misha Gromov's notation $\mathcal{O}p(A)$ or $\mathcal{O}p_X(A)$ for an unspecified open neighborhood of a subset~$A$ in a topological space $X$.

\subsection{PL symplectic topology}

A \textsl{Whitney form of degree $k$} or a \textsl{PS $k$-form} on a simplicial complex $K$ is a collection $\alpha = (\alpha_\sigma)_{\sigma \in K}$, where $\alpha_\sigma$ is a smooth $k$-form on $\sigma$, such that 
$${\alpha_{\sigma_1}}\Bigl|_{\sigma_1 \cap \sigma_2} = {\alpha_{\sigma_2}}\Bigl|_{\sigma_1 \cap \sigma_2},$$
for all $\sigma_1, \sigma_2 \in K$ (see~\cite{sullivan} Section 7). Of course, if $K'$ is a refinement of $K$, the form $\alpha$ induces a Whitney $k$-form, also denoted by $\alpha$, on $K'$. A \textsl{Whitney form} of degree $k$ on a PL manifold $P$ of dimension $n$ is a Whitney $k$-form on a simplical complex $K$ that triangulates $P$. Observe that the defining condition above is empty in case $k = n$. A Whitney $n$-form is thus simply a collection of smooth $n$-forms on the simplices of dimension $n$.

\medskip

If $\alpha$ is a smooth $k$-form on a smooth manifold $M$ of dimension $n$ endowed with a Whitehead triangulation $h : |K| \to M$, then the pullback of $\alpha$ through $h$ is defined to be the Whitney $k$-form $h^*\alpha = (h|_{\sigma}^*\alpha)_{\sigma \in K}$.

\medskip

The \textsl{differential} of a Whitney $k$-form $\omega$ is the Whitney $(k+1)$-form defined by $d\omega = (d\omega_\sigma)_{\sigma \in K^{(n)}}.$ A Whitney $k$-form $\omega$ is called \textsl{exact} if $\omega = d\alpha$ for some Whitney $(k-1)$-form $\alpha$. We can define a cohomology associated to piecewise smooth forms on $K$. It is isomorphic to the simplicial cohomology of $K$ (cf. \cite{sullivan}). 

\medskip

Two Whitney forms $\alpha_1$ and $\alpha_2$ on a simplicial complex $K$ are said to be PS (respectively PL) equivalent when there exists two subdivisions $K'$ and $K''$ of $K$ and a PS (respectively PL) homeomorphism $h : K' \to K''$ such that $h^*\alpha_1=\alpha_2$. The notion of equivalence between $k$-forms on a PL manifold is defined accordingly.

\begin{df} Let $P$ denote a PL manifold of dimension $n$.

$\bullet$ A \textsl{piecewise constant (PC) $k$-form} on $P$ is a Whitney $k$-form $\omega = (\omega_\sigma)_{\sigma \in K}$ on some triangulation $K$ of $P$ such that $\omega_\sigma$ is a constant form for each $\sigma \in K$. \par

$\bullet$ For $k = n$, a Whitney $n$-form on $P$ is a  \textsl{volume form} if, for each $\sigma \in K^{(m)}$, the form $\om_\sigma$ does not vanish and if the $\om_\sigma$'s induce an orientation of $P$.\par

$\bullet$ For $k = 2$ and $n = 2m$, a Whitney form on $P$ is said to be symplectic if $\omega_\sigma$ is closed and non-degenerate for each $\sigma \in K^{(n)}$ and $\om^m = (\om^m_\sigma)_{\sigma \in K}$ is a volume form on $P$.

$\bullet$ If $P$ is oriented and endowed with a Whitney $n$-form $\Omega = (\Om_\sigma)_{\sigma \in K^{(n)}}$, the integral of $\Omega$ along $P$ is denoted by $\int_P\Omega$ and defined to be the sum
$$\sum_{\sigma \in K^{(n)}} \int_\sigma \Omega|_\sigma.$$ 
The \textsl{volume cocycle of $\Omega$} is the map $V_\Om : \sigma \in K^{(n)} \to \int_\sigma \Omega|_\sigma$.\\

\end{df}
Observe that the condition to be constant for a $k$-form defined on a simplex $\sigma=\langle v_0, \ldots, v_\ell \rangle$ is well-defined since $\Lambda^kT^*\sigma$ has a canonical class of trivializations for which the collections of constant forms are all the same.

\begin{df}\label{PCvol} Given a closed smooth form $\Omega$ on a smooth manifold $M$, a \emph{triangulation of $\Omega$} is a Whitehead triangulation $h : |K| \to M$ (that is $h|_\sigma$ is a smooth embedding for all $\sigma \in K$) such that $h^*\Omega$ is piecewise constant.
\end{df}

\section{Triangulations of volume forms}

Consider an oriented smooth manifold $M$ of dimension $n \geq 2$, equipped with a positive volume form $\Omega$. If $h : |K| \to M$ is a smooth triangulation of $M$, consider the Whitney $n$-form $\Omega_0 = h^*\Omega$ and some piecewise constant $n$-form $\Omega_1$ on $K$. As explained in the introduction, we would like to select $\Omega_1$ and $h$ so that a primitive of $\Omega_0 - \Omega_1$ may be chosen that vanish along the skeleton. This ensures that the flow of the time-dependent vector field $X_t$ defined by $i(X_t)\Omega_t = \alpha$, for $\Omega_t = \Omega_0 + t (\Omega_1 - \Omega_1)$, is stationary along the skeleton and induces therefore a well-defined PS isotopy $(\varphi_t)$ of $M$ which will achieve $\varphi_1^*\Omega_0 = \Omega_1$. (Observe that $\alpha|_{\partial \sigma} = 0$ is not sufficient, as the latter condition only implies that $X_t$ is tangent to $\partial \sigma$.) Let us investigate now the precise conditions to be satisfied by $\Omega_1$ and $h$.

\begin{enumerate}

\item In order to be able to apply Moser's argument to volume forms $\Omega_0$ and $\Omega_1$ on an oriented $n$-simplex $\sigma$, it is necessary for them to be cohomologous, that is, to have the same total volume~:
$\int_{\sigma} \Omega_1 = \int_{\sigma} \Omega_0.$ In other words, $\Omega_0$ and $\Omega_1$ must have identical volume cocycle.

\item Suppose that there exists a diffeomorphism $\varphi$ of $\sigma$ which coincides with the identity on $\partial \sigma$ and such that $\varphi^*({\Omega_1}|_\sigma) = {\Omega_0}|_\sigma$. Then $\varphi_{*_x} = {\rm Id}$ at any point $x$ of the codimension-two skeleton since, at such a point, the tangent space to $\sigma$ is generated by the tangent spaces to the faces passing through $x$. Therefore, a second necessary condition is that the forms $\Omega_0$ and $\Omega_1$ coincide at points of $\sigma^{(n-2)}$.

\end{enumerate}

Given some PS volume form $\Omega_0 = h^*\Omega$, a piecewise constant form $\Omega_1$ that satisfies the first condition is easily constructed~: for each $n$-simplex $\sigma$ endowed with the orientation induced by $\Omega_0$, let $\Omega_1^\sigma$ denote the constant $n$-form on $\sigma$ whose total volume is $\int_\sigma \Omega_0$. These forms gather into a piecewise constant volume form $\Omega_1 = {\mathcal C}_K(\Omega_0) = (\Omega_1^\sigma)_{\sigma \in K^{(n)}}$ on $K$ that has the same volume cocycle as $\Omega_0$. \\

The rest of this chapter consists in constructing a triangulation $h' : |K'| \to M$, defined on some refinement $K'$ of $K$ for which $h'^*\Omega$ and $\Omega_1$ agree near the skeleton of $K'$. 
This is done by induction on the dimension of the simplices and allows us to construct a primitive of $\Omega_0 - \Omega_1$ that vanishes in a neighborhood of the skeleton. We may then use Moser's argument, which we briefly recall now, inside each $n$-simplex .

\begin{prop}\label{mo} [Relative Moser]
Let $\Omega_0$ and $\Omega_1$ be smooth volume forms on an oriented $n$-simplex $\sigma$ that agree near its boundary and have the same volume. Then there exists a diffeomorphism $\varphi$ of $\sigma$ such that $\varphi^*\Omega_1 = \Omega_0$ and $\varphi|_{\mathcal{O}p(\partial \sigma)} = \mathrm{Id}$.
\end{prop}

\begin{proof}
Since the $n$-form $\Omega_0 - \Omega_1$ is compactly supported in $\rm Int(\sigma)$ and has volume zero, it represents the trivial cohomology class in $H^n_c({\rm Int} (\sigma))$. It admits therefore a compactly supported primitive $\alpha$. The time-dependent vector field $X_t$ defined by  
$$
\iota(X_t)(\Omega_0 + t d\alpha) = \alpha
$$
has compact support. Its flow provides us with an isotopy $\varphi_t$ that is stationary near $\partial \sigma$ and satisfies $\varphi_t^*\Omega_t = \Omega_0$ for each $t \in [0,1]$.
\end{proof}

\begin{rem}
If we only assume that $\Omega_1$ and $\Omega_0$ coincide at points of the codimension two skeleton, it is also possible to construct a primitive $\alpha$ of $\Omega_1-\Omega_0$ which vanishes on $\partial \sigma$. This result is shown in \cite{bruveris}. More precisely, the authors prove that if $\Omega_0$ and $\Omega_1$ have same total volume on a $n$-simplex $\sigma$ and agree on $\sigma^{(n-2)}$ then a primitive that vanishes along $\partial \sigma$ can be constructed. Of course this hypothesis has no reason to be satisfied here and the main body of our argument consists in showing that $\Omega_0$ is PS equivalent to a volume form $\Omega'_0$ that coincides with $\Omega_1$ along the codimension-two skeleton. We prove more~: these forms can be made to agree on a neighborhood of the skeleton. Their difference admits therefore a primitive vanishing there. 
\end{rem}

\subsection{Agreement near the skeleton}\label{mod2}

We are given a smooth volume form $\Omega$ on a manifold $M$ endowed with a triangulation $h : |K| \to M$. We consider the PS form $\Omega_0 = h^*\Omega$ and the PC form $\Omega_1 = {\mathcal C}_K(\Omega_0)$ on $K$. The purpose of this section is to replace the triangulation $h : |K| \to M$ by $\widetilde{h}$ of the type $\psi \circ h \circ \varphi$, where $\varphi$ is a PS homeomorphism of some refinement ${K'}$ of $K$ and $\psi$ is a diffeomorphism of $M$, such that $\widetilde{h}^*\Omega$ coincides with $\Omega_1$ near the skeleton of the new simplicial complex ${K'}$. We first handle the $0$-skeleton $K^{(0)}$ and then treat successively the higher-dimensional skeleta by means of parametric versions of the $0$-skeleta case. Obtaining such a parametric versions is the subtle part of the argument.

\subsubsection{First step}

Fix a $n$-simplex $\sigma$. The first step is to deform $\Omega_0$ via a piecewise smooth homeomorphism of $\sigma$ to make it equal to $\Omega_1$ near $|\sigma^{(0)}|$. The idea is to consider the stellar subdivision $\mathcal{S}\sigma$ of $\sigma$ and to construct a piecewise smooth homeomorphism of $\mathcal{S}\sigma$ which fixes $\partial \sigma$ and moves the points in a neighborhood of $v \in \sigma^{(0)}$ in the directions of the edge $<v, b>$. As the property that the volumes of $\Omega_0$ and $\Omega_1$ agree on the $n$-simplices has no reason to be true on $\mathcal{S}\sigma$, we first attend to that and, at the same time, take care of the new vertex~$b_\sigma$. 

\begin{lemma}\label{subd}
Let $\Omega_0$ and $\Omega_1$ be smooth volume forms on a linear simplex $\sigma$ with identical total volume and let $\mathcal{S}\sigma$ be the stellar subdivision of $\sigma$. There exists a smooth volume form $\overline\Omega$ on $\sigma$ and a diffeomorphism $\varphi$ of $\sigma$ such that
\begin{itemize}
\item $\varphi^*\Omega_0 = \overline\Omega$ and $\varphi$ fixes $\mathcal{O}p(\partial \sigma)$,
\item $\int_{\tau} \overline\Omega = \int_{\tau} \Omega_1$ for all $\tau \in \mathcal{S}\sigma^{(n)}$,
\item $\overline\Omega = \Omega_1$ near $b_\sigma$.
\end{itemize}
\end{lemma}

\begin{proof}
Define $\overline\Omega=f\Omega_0$, where $f$ is a smooth positive function on $\sigma$ which is identically equal to one near $\partial \sigma$, and such that $f\Omega_0$ and $\Omega_1$ have the same volume cocycle on $\mathcal{S}\sigma$ and agree near the new vertex $b_\sigma$. Now observe that $\overline\Omega$ and $\Omega_0$ satisfy the assumptions of Proposition~\ref{mo}, so that there exists a diffeomorphism $\varphi$ of $\sigma$ which coincides with the identity near $\partial \sigma$ and for which $\varphi^*\Omega_0 = \overline\Omega$.
\end{proof}

The next step is a deformation of $\sigma$ which makes $\overline\Omega$ equal to $\Omega_1$ near the vertices of $\sigma$. Notice that the proof of the following lemma uses a technical result which is shown later, in Section \ref{interpolation}.

\begin{lemma}\label{vert}
Consider smooth volume forms $\overline\Omega$ and $\Omega_1$ on a linear simplex $\sigma$ inducing the same orientation.
Then there exists a PS homeomorphism $\rho$ of $\mathcal{S}{\sigma}$ such that
\begin{itemize}
\item $\rho$ is the identity on $\sigma\backslash \mathcal{O}p(|\sigma^{(0)}|)$,
\item $\rho$ preserves $|\mathcal{S}\sigma^{(n-1)}|$ and fixes $\partial \sigma$,
\item $\rho^*\overline\Omega = \Omega_1$ on $\mathcal{O}p(|\sigma^{(0)}|)$,
\item If $\overline\Omega$ and $\Omega_1$ coincide near a vertex $v$, then $\rho$ is the identity near $v$.
\item The forms $\overline\Omega$ and $\rho^*\overline\Omega$ have identical volume cocyles on $\mathcal{S}{\sigma}$.
\end{itemize}
If, in addition, $\overline\Omega$ and $\Omega_1$ depend smoothly on a parameter $\lambda$ running in some smooth manifold $\Lambda$, then so does $\rho$. 
\end{lemma} 

\begin{proof}
Fix a vertex $v \in \sigma^{(0)}$ and denote by $b$ the barycenter of $\sigma$. Let $\partial_v \sigma$ denote the boundary of $\sigma$ minus the interior of the $(n-1)$-face opposite to $v$, and let $\pi$ be the projection of $\sigma$ onto $\partial_v \sigma$, parallel to the edge $<b,v>$. Observe that any point $x$ in $\sigma$ can be written as 
$$
x = \pi(x) +t(x)(b-v)
$$
for a certain $t(x) \in \R^+$. 
Now let $F$ be the smooth positive function defined by the relation $\overline\Omega = F\Omega_1$ and define a piecewise smooth map $\varphi_F$ on $\sigma$ (smooth on each $\tau \in \ES \sigma$), with values in the affine plane spanned by $\sigma$, by
$$
\varphi_F(x)=\pi(x)+\left(\int_0^{t(x)} F(\pi(x)+s(b-v))ds \right)(b-v).
$$
Clearly, $\varphi_F$ coincides with the identity on $\partial_v \sigma$. If $\tau$ is a $(n-1)$-face of $\sigma$ which contains $v$, we can choose linear coordinates $x_1, \ldots, x_{n-1}$ in $\tau$ and complete them with $t$ to obtain coordinates on $\sigma \cap \pi^{-1}(\tau)$. Then
$$
\varphi_F(x_1,\ldots,x_{n-1},t) =\left(x_1, \ldots, x_{n-1}, \int_0^{t} F(x_1, \ldots, x_{n-1}, s) ds \right)
$$
on $\sigma \cap \pi^{-1}(\tau)$. Whence $|\mathrm{Jac}(\varphi_F)|=F$ and $\varphi_F^*\Omega_1 = \overline\Omega$, on all of $\sigma$. \\

The idea is now to modify $\varphi_F$ away from $v$ in such a way that it becomes the identity out of an arbitrary small neighborhood of $v$. In the affine plane spanned by $\sigma$, consider a ball $B$ centered at $v$ that does not contain $b$. Choose also a non-negative function $\chi$ on $\sigma$ that vanishes near $v$, equals $1$ outside $\frac{1}{2}B$ (the ball centered at $v$ of half the radius) and that is non-decreasing along the radial directions. Define $\overline{F}=(1-\chi)F+\chi$ and consider the map $\varphi_{\overline{F}}$. Since $\overline{F}$ equals $1$ out of $\frac12 B$, any point $x$ such that $\pi(x) \notin \frac12B$ satisfies $\varphi_{\overline{F}}(x)=x$. Altogether $\varphi_{\overline{F}}$ satisfies the assumptions of Lemma \ref{inter2} below, and there exists thus a piecewise smooth homeomorphism $\overline\varphi_{\overline{F}}$ of $\ES\sigma$ which coincides with $\varphi_{\overline{F}}$ near $v$ and with the identity near $\partial B$. The same procedure can be repeated near each vertex and this completes the proof.

Moreover, the procedure just described, being explicite and relying on Lemma \ref{inter2}, provides us with a smooth dependence of $\rho$ on a parameter $\lambda$ whenever the initial forms $\overline\Omega$ and $\Omega_1$ depend smoothly on $\lambda$.
\end{proof} 

Thus, given a smooth volume form $\Omega$ on a manifold $M$ endowed with a triangulation $h : |K| \to M$ and the associated forms $\Omega_0 = h^*\Omega$ and $\Omega_1 = {\mathcal C}_K(\Omega_0)$, Lemma \ref{subd} and Lemma \ref{vert} provide, for each $n$-simplex $\sigma$, piecewise smooth homeomorphisms $\varphi_\sigma$ and $\rho_\sigma$ of $\ES \sigma$ which, among may other properties, are the identity on $\partial \sigma$. When $\sigma$ runs through $K^{(n)}$, the maps $\varphi_\sigma$ and $\rho_\sigma$ form thus global piecewise smooth homeomorphisms of $\ES K$, denoted respectively by $\varphi$ and $\rho$. The associated triangulation $\tilde{h}=h \circ \varphi \circ \rho : |\mathcal{S}K| \to M$ satisfies the relation
$$
\tilde{h}^*\Omega = \rho^*(\varphi^*(h^*\Omega)) = \rho^*(\varphi^*\Omega_0) = \rho^*\overline{\Omega}.
$$
The latter form coincides with the piecewise constant form $\Omega_1$ near the vertices of $\mathcal{S}K$ and admits the same volumes cocyle as $\Omega_1$ on $\mathcal{S}K$. We have thus proven the following result.

\begin{prop}\label{conclusion1}
Given a smooth volume form $\Omega$ on a manifold $M$, there exists a triangulation $h : |K| \to M$ 
such that $h^*\Omega$ and $\Omega_1 =  {\mathcal C}_K(\Omega_0)$ have the same volume cocycle on $K$ and coincide near the vertices of $K$. 
\end{prop}

\subsubsection{Induction step}

We will henceforth assume that $h : |K| \to M$ and $\Omega_1$ satisfy the conclusion of Proposition~\ref{conclusion1}. We fix again a $n$-simplex $\sigma$. The goal now is to extend the area where $\Omega_1$ and $h^*\Omega$ agree from neighborhood of the vertices to neighborhoods of $|\sigma^{(1)}|$, $|\sigma^{(2)}|$, $\ldots$, $|\sigma^{(n)}|$ successively. The procedure is the same as for the $0$-skeleton~: assuming that $\Omega_0 = \Omega_1$ on $\mathcal{O}p(|\sigma^{(k-1)}|)$, we first show that $\Omega_0$ can be replaced by a form which admits the same volume cocycle as $\Omega_1$ on $\mathcal{S}\sigma$, and which coincides with $\Omega_1$ on a neighborhood of the new $k$-simplices (those in $\mathcal{S}\sigma^{(k)} \setminus \sigma^{(k)}$). We then handle the $k$-simplices of $\sigma$.

\begin{lemma}\label{subd2}
Let $\Omega_0$ and $\Omega_1$ be smooth volume forms of equal total volume on a linear simplex $\sigma$. Assume moreover that $\Omega_1 = \Omega_0$ on a neighborhood of $|\sigma^{(k-1)}|$ for a certain $k\in\{1,\ldots, n\}$. 
Then there exists a smooth volume form $\overline\Omega$ and a diffeomorphism $\varphi$ of $\sigma$ such that
\begin{itemize}
\item $\varphi^*\Omega_0 = \overline\Omega$ and $\varphi$ fixes $\mathcal{O}p(\partial \sigma)$,
\item $\int_{\tau} \overline\Omega = \int_{\tau} \Omega_1$ for any $\tau \in \mathcal{S}\sigma^{(n)}$,
\item $\overline\Omega = \Omega_1$ on a neighborhood of $|\mathcal{S}\sigma^{(k)}\backslash \sigma^{(k)}|$.
\end{itemize}
\end{lemma}

\begin{proof}
The proof is exactly the same as that of Lemma \ref{subd}. The additional assumption that $\Omega_1 = \Omega_0$ near $|\sigma^{(k-1)}|$ allows us to choose a function $f$ that achieves the value one near $\partial \sigma$ and satifies $f\Omega_0 = \Omega_1$ near $|\mathcal{S}\sigma^{(k)}\backslash \sigma^{(k)}|$, in addition to the condition that $\int_\tau f\Omega_0 = \int_\tau \Omega_1$ for each $\tau \in \mathcal{S}\sigma^{(n)}$.
\end{proof}

Now fix a $k$-simplex $\tau$ of $\sigma$ and assume that $\overline\Omega = \Omega_1$ near $\partial \tau$. The idea, to extend the equality to a neighborhood of $\tau$, is to consider a collection of parallel $(n-k)$-planes along $\tau$, transverse to $\tau$, and to apply Lemma \ref{vert} in each of them. 

\begin{lemma}\label{extension}
Consider smooth volume forms $\overline\Omega$ and $\Omega_1$ on a linear simplex $\sigma$ such that $\overline\Omega = \Omega_1$ on $\mathcal{O}p(|\sigma^{(k-1)}|)$ for a certain $k\in\{1, \ldots, n\}$. Then there exists a piecewise smooth homeomorphism $\rho$ of $\mathcal{S}{\sigma}$ such that
\begin{itemize}
\item $\rho$ is the identity on $\sigma\backslash \mathcal{O}p(|\sigma^{(k)}|)$ and fixes $\partial \sigma$,
\item $\rho^*\overline\Omega = \Omega_1$ on $\mathcal{O}p(|\sigma^{(k)}|)$,
\item $\rho^*\overline\Omega$ and $\Omega_1$ have the same volume cocyle on $\mathcal{S}{\sigma}$.
\end{itemize}
\end{lemma}

\begin{proof}
Let $v_0, \ldots, v_n$ denote the vertices of $\sigma$. Consider the $k$-face $\tau$ of $\sigma$ generated by the vertices $v_{i_0}, \ldots, v_{i_k}$, where $0 \leq i_0 <\ldots < i_k \leq n$. On a neighborhood $U$ of $\tau$ in $\sigma$, define the constant vector fields 
$$
X_l = v_l - v_{i_0}
$$
for any $l \neq i_0$. Observe that $X_l$ is tangent to $\tau$ if and only if $v_l$ is a vertex of $\tau$. For any $x \in \mathrm{int}(\tau)$, consider the $(n-k)$-dimensional affine plane $\mu_x$ passing through $x$ and generated by the $X_l$'s with $l \notin \{i_1, \ldots, i_k\}$. If $b$ denotes the barycenter of $\sigma$, let $\nu$ be the $(k+1)$-simplex $\nu$ of $\mathcal{S}\sigma$ generated by $\tau$ and $b$, and notice that the intersection $\nu \cap \mu_x$ is a segment for any $x\in \mathrm{int}(\tau)$. Now consider the $(n-k)$-forms
\begin{align*}
\overline{\omega} & = \iota(X_{i_1},\ldots, X_{i_k}) \; \overline\Omega, \\
\omega_1 & = \iota(X_{i_1},\ldots, X_{i_k}) \; \Omega_1.
\end{align*}
Notice that their restrictions to the simplices $\sigma \cap \mu_x$, $x \in \mathrm{int}(\tau)$, are volume forms. Now, for each $x \in \mathrm{int}(\tau)$, we apply Lemma \ref{vert} to obtain a PS homeomorphism $\rho_x$ of $\sigma  \cap \mu_x$ which is the identity outside of $U$, moves points in the direction of the segment $\nu \cap \mu_x$, fixes $\partial \sigma \cap \mu_x$ and satisfies $(\rho_x)^*\overline{\omega} = \omega_1$ in $\mathcal{O}p(\tau) \cap \mu_x$. We may also assume that the family of maps $\rho_x$ depends smoothly on $x \in {\rm Int} (\tau)$ and coincides with the identity map for $x$ near~$\partial \tau$. 

Notice that although we may not strictly speaking directly apply Lemma \ref{vert} because the simplices $\sigma  \cap \mu_x$ vary, we could reparameterize them affinely by a standard simplex $\Delta_0$, work there and push-forward the resulting maps to $\sigma \cap \mu_x$. 

This procedure provides us with a smooth map $\rho_\tau : \sigma \to \sigma$ which fixes $\tau$ and $\partial \sigma$, preserves each simplex in $\mathcal{S}\sigma$, is the identity outside $U$ and satisfies $\rho_\tau^*\overline{\Omega} = \Omega_1$ near $\tau$. Finally, the various $\rho_\tau$ for $\tau \in \sigma^{(k)}$ may be glued into a PS homeomorphism $\rho$ of $\mathcal{S}\sigma$ that satisfies the desired properties.
\begin{figure}[htb]
  \centering
  \def\svgwidth{180pt}
  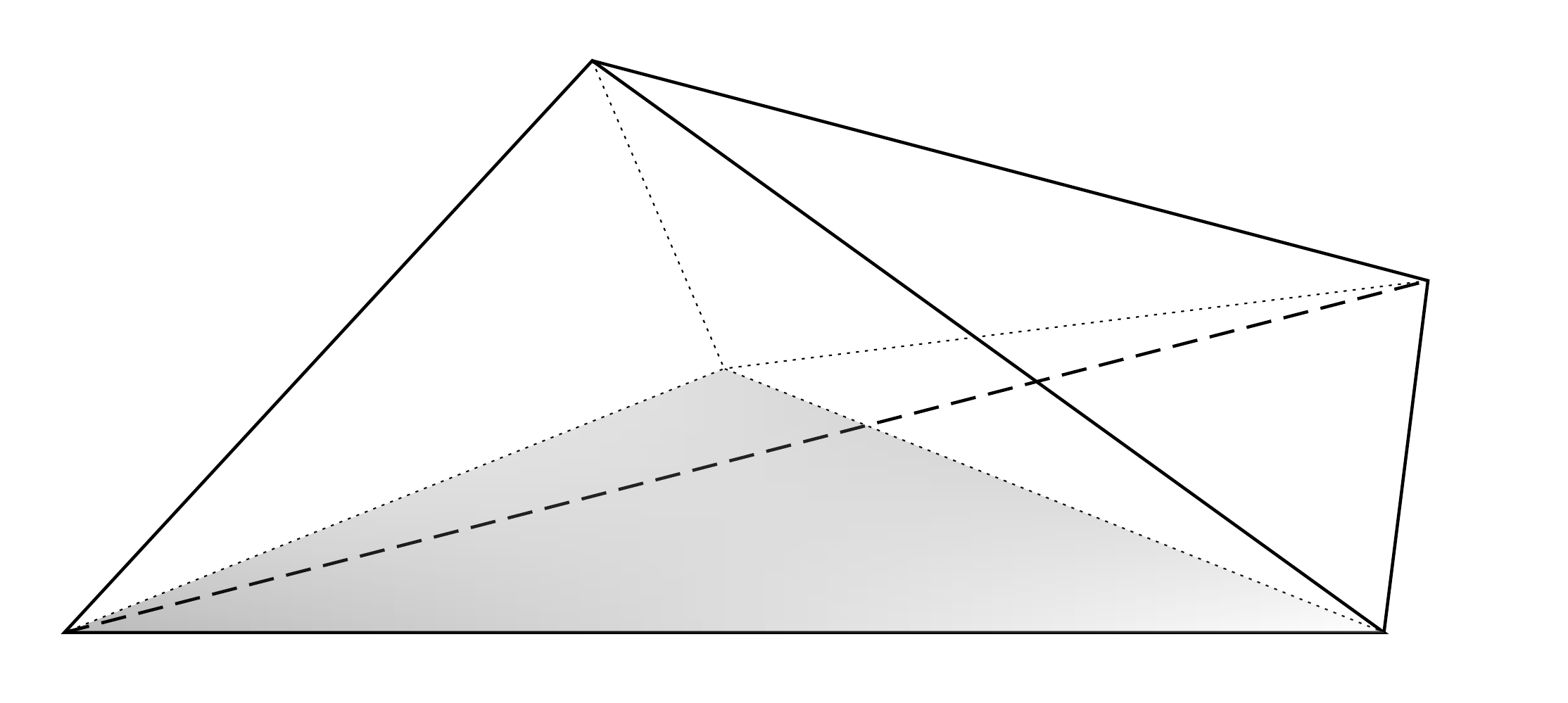
\caption{Stellar subdivision of a 3-simplex and the simplex $\nu$.}
\end{figure}
\end{proof}

Notice that, for $k = n-1$, it is not necessary to subdivide $\sigma$. Once again, the piecewise smooth homeomorphisms resulting from Lemmas \ref{subd2} and \ref{extension} 
fix the boundaries of the $n$-simplices and thus glue into global piecewise smooth homeomorphisms of $K$. The following proposition summarizes what has been done so far.

\begin{prop}\label{conclusion2}
let $h : |K|\to M$ be a smooth triangulation. Consider a smooth volume form $\Omega$ on $M$, and a piecewise smooth volume form $\Omega_1$ on $K$ such that $h^*\Omega$ and $\Omega_1$ have the same volume cocycle. Assume that $h^*\Omega$ coincides with $\Omega_1$ on a neighborhood of $|K^{(k-1)}|$ for a certain $1 \leq k \leq n$. Then there exists a smooth triangulation $\ol{h} : |\mathcal{S}K| \to M$ for which $\ol{h}^*\Omega$ and $\Omega_1$ have the same volume cocycles on $\mathcal{S}K$ and coincide on a neighborhood of $|\mathcal{S}K^{(k)}|$. 
\end{prop}

\subsection{Interpolation}\label{interpolation}

In this section we prove Lemma \ref{inter2} needed in the proof of Lemma \ref{vert}. The fundamental ingredient is an explicit procedure to build from a smooth increasing function $\varphi : \R^+ \to \R^+$ with $\varphi(0) = 0$, a smooth increasing interpolation $\ol{\varphi}$ between $\varphi$ near the origin and the identity near $R$. This result is of course far from surprising but the point is the explicit character of the construction which yields a parametric version of the interpolation process.

To build such an interpolation, we use the convolution product with regularizing functions. Consider the smooth even function
$$
\rho_0 : \R \to \R^+ : x \mapsto \rho_0(x)=\left\{
\begin{array}{ll}
e^{\left(-\frac{1}{1-x^2}\right)} &\mathrm{if}~~ |x| < 1\\
0 & \mathrm{if}~~ |x| \geq 1.
\end{array}
\right.
$$
Let us denote its integral by $A$ and define the functions $\rho$ and $\rho_\delta$, for $\delta >0$, by $\rho = \frac{1}{A} \rho_0$ and $\rho_\delta(x) = \frac{1}{\delta}\rho(\frac{x}{\delta})$ respectively. Then $\int \rho_\delta = 1$ for all $\delta$. 

As is well known, the convolution product $f * \rho_\delta$ of a continuous function $f : \R \to \R$ and $\rho_\delta$ is a smooth function that $C^0$ converges to $f$ when $\delta$ tends to $0$. The following basic properties will be useful later on :
\begin{itemize}
\item If $f$ is an increasing function, then $f*\rho_\delta$ is also increasing. Indeed, observe that 
\begin{align*}
(f*\rho_\delta)'(x)&=\int_{-\delta}^{\delta}\rho'_\delta(y)f(x-y)dy\\
&=\int_{0}^{\delta}\rho'_\delta(y)\left(f(x-y)-f(x+y)\right)dy
\end{align*}
which is non-negative since both $\rho'_\delta(y)$ and $f(x-y) - f(x+y)$ are non-positive for any $y\geq0$.
\item If $f$ is a continuous piecewise linear function of the following simple type~:
$$
f(x)=\left\{
\begin{array}{ll}
ax+b &\mathrm{if}~~ x < \epsilon \\
cx+d & \mathrm{if}~~ x \geq \epsilon
\end{array}
\right.
$$
for some $a, b, c, d, \epsilon \in \R$ with $a\epsilon + b = c \epsilon + d$, 
then $(f*\rho_\delta)(x)=f(x)$ if $|x - \epsilon| \geq \delta$. Indeed, if $x \geq \epsilon+\delta$, then $x-y\geq \epsilon$ for any $y \in [-\delta,\delta]$, so that
\begin{align*}
(f*\rho_\delta)(x) &= \int_{-\delta}^{\delta} f(x-y)\rho_\delta(y)dy \\
&=\int_{-\delta}^{\delta}\bigl( c(x-y) + d \bigr)\rho_\delta(y)dy\\
&=(cx+d)\int_{-\delta}^{\delta}\rho_\delta(y)dy -c \int_{-\delta}^{\delta}y\rho_\delta(y)dy\\
&=cx+d,
\end{align*}
since $y\rho_\delta(y)$ is an odd function. The case $x \leq \epsilon-\delta$ is close to identical.
\end{itemize}

\begin{lemma}\label{inter}
Consider a smooth non-decreasing map $\varphi : [0,R] \to \R^+$ such that $\varphi(0)=0$. Then there exists $r > 0$ and a smooth non-decreasing map $\overline{\varphi} : [0,R] \to \R^+$ such that 
\begin{itemize}
\item $\overline{\varphi}$ coincides with $\varphi$ on $[0,r]$,
\item $\overline{\varphi}(x)=x$ for any $x\in [R-r,R]$,
\end{itemize}
Moreover, if $\varphi=\mathrm{Id}$, then $\overline{\varphi}=\mathrm{Id}$. Furthermore, if $\lambda \to R_\lambda$ is a smooth map defined on some manifold $\Lambda$ and if $\varphi_\lambda : [0, R_\lambda] \to \R^+$ is a smooth family of maps enjoying the same properties as $\varphi$, then $r = r(\lambda)$ and $\overline{\varphi}_\lambda$ depend smoothly on $\lambda$ as well.
\end{lemma}

\begin{figure}[htb]
  \centering
  \def\svgwidth{180pt}
  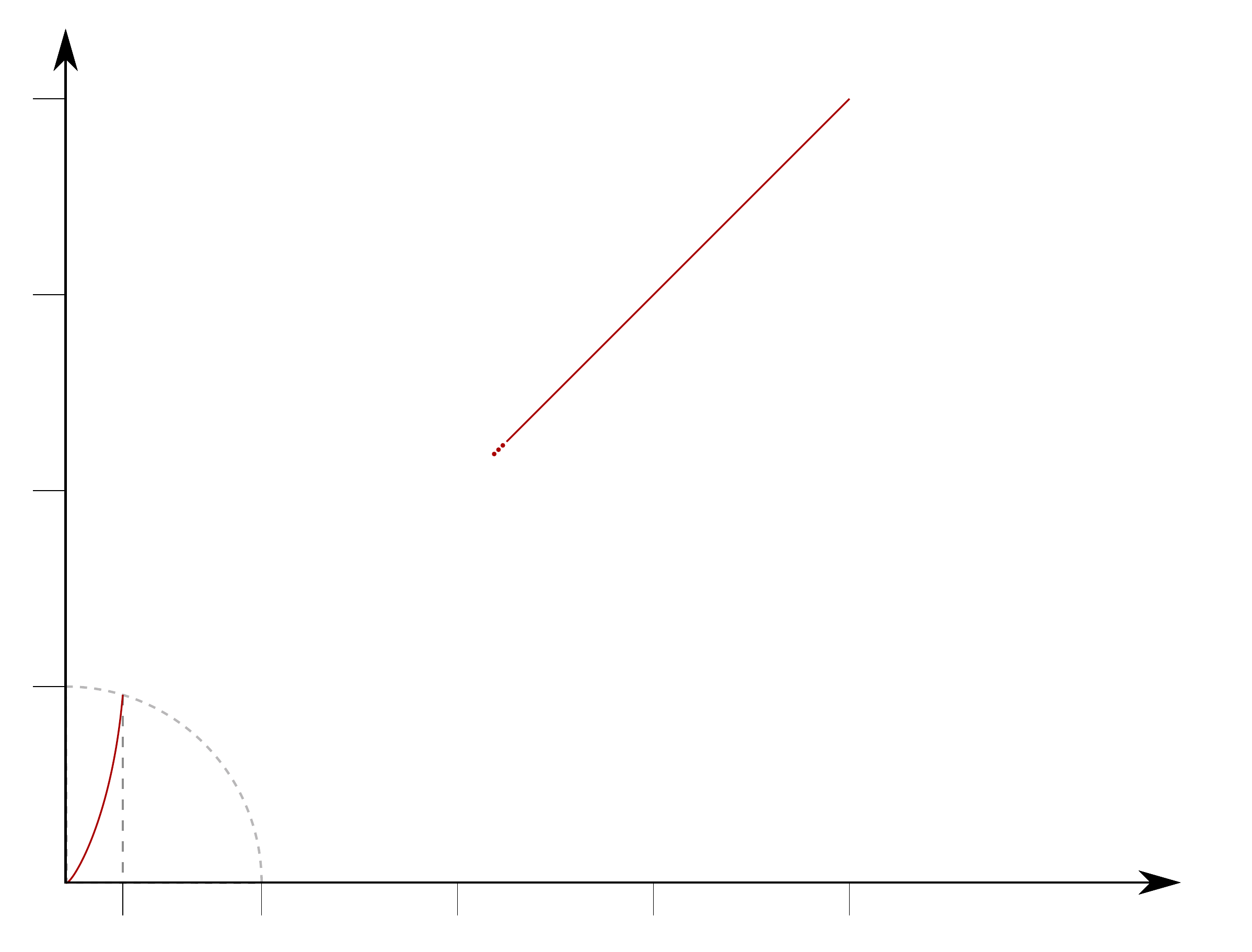
\end{figure}

\begin{proof}
Consider the graph of $\varphi$ in $\R^2$ and fix $\epsilon >0$ such that the point $(\epsilon, \varphi(\epsilon))$ is at distance $\frac R4$ from the origin. Of course, the number $\epsilon$ depends smoothly on $R$. Define $g: [\epsilon, \frac{3R}{4}] \to \R$ to be the affine map whose graph is the segment joining $(\epsilon, \varphi(\epsilon))$ to $(\frac{3R}4,\frac{3R}4)$, and let $h$ be the piecewise smooth function given by 
$$
h(x)=\left\{
\begin{array}{ll}
\varphi(x) &\mathrm{if}~~ x\in[0,\epsilon],\\
g(x) & \mathrm{if}~~ x\in ]\epsilon,\frac{3R}{4}[,\\
x &\mathrm{if}~~ x\in[\frac{3R}4,R].
\end{array}
\right.
$$

\begin{figure}[htb]
  \centering
  \def\svgwidth{180pt}
  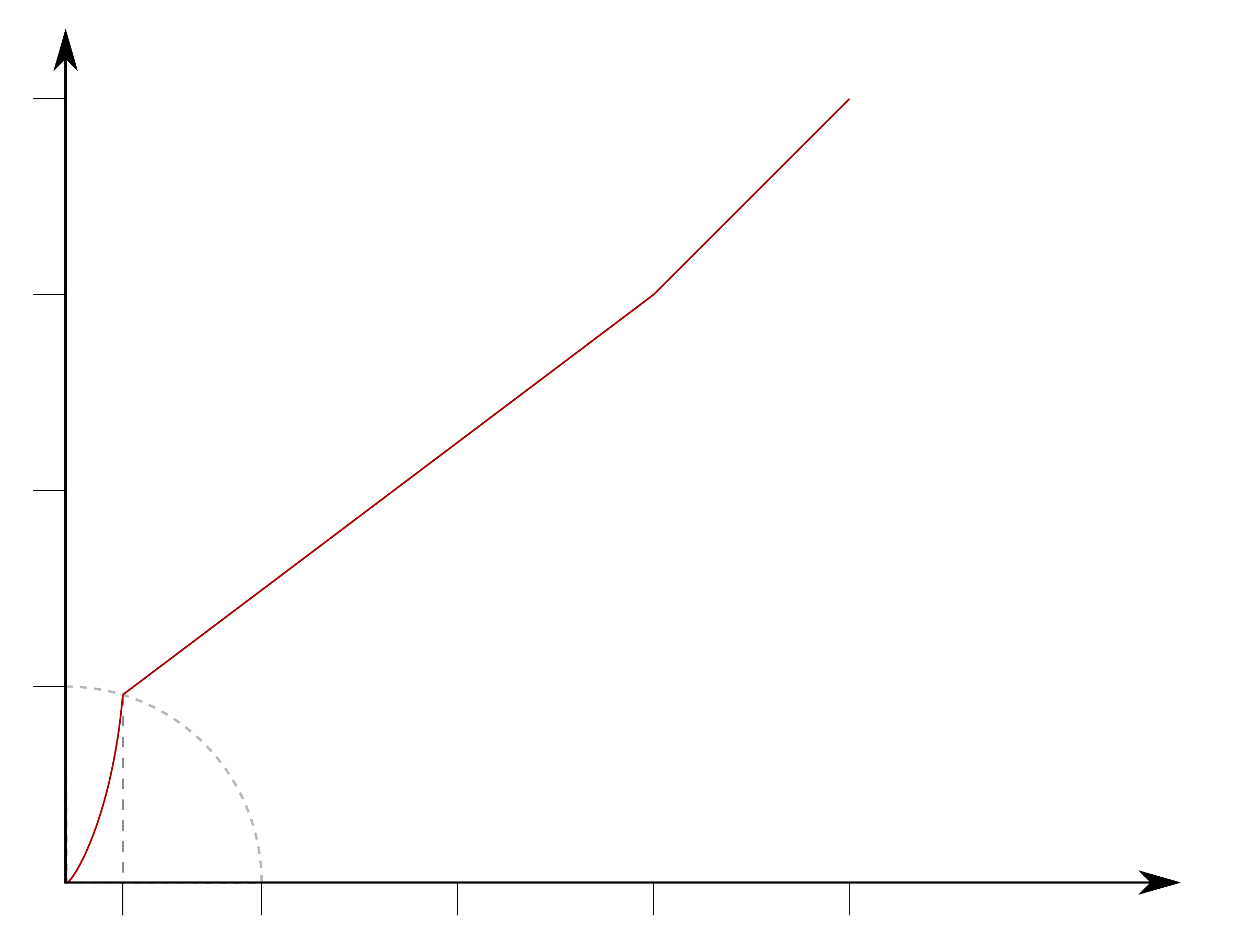
\caption{The piecewise smooth map $h$.}\label{h2}
\end{figure}

This map provides a non-decreasing interpolation between $\varphi$ and the identity, but it is smooth neither at $x=\epsilon$ nor at $x = \frac{3R}4$ (cf. Figure \ref{h2}). Now fix $\delta=\frac{\epsilon} 2$ and consider the regularizing function $\rho_\delta$ defined previously. If $x \in [\frac{R}2,R]$, define $\overline{\varphi}(x)=(h*\rho_\delta)(x)$. Since $h$ is piecewise linear and non-decreasing on $[\frac{R}2,R]$, the function $\overline{\varphi}$ is a smooth monotone map which coincides with $h$ outside the interval $(\frac{3R}4-\delta,\frac{3R}4+\delta)$. On the interval $[0,\frac{R}2)$, the idea is also to replace $h$ by the convolution product $h*\rho_\delta$, but we have to modify the convolution slightly so that $h$ and $h*\rho_\delta$ can be smoothly glued. Let $\chi$ be a smooth non-decreasing function such that $\chi(x)=0$ on $[0,\frac{\epsilon}3]$, $\chi(x)=1$ if $x \geq \frac{2\epsilon}3$ and which depends smoothly on $\epsilon$. If $x\in[0,\frac{R}2]$, define
$$
\overline{\varphi}(x)=\int_{-\delta}^{\delta}h\bigl(x-\chi(x)y\bigr)\rho_\delta(y)dy.
$$
Then $\overline{\varphi}(x)$ is the convolution product of $h$ and $\rho_\delta$ for $x\geq\frac{2\epsilon}3$. Moreover $\overline{\varphi}(x) = h(x)$ when $x \in [0,\frac{\epsilon}3]$. Observe finally that $\overline{\varphi}$ is increasing by a similar argument as for the convolution product. 
Now if $\varphi=\mathrm{Id}$, then $h=\mathrm{Id}$ by construction and, therefore $\overline{\varphi}(x)=\mathrm{Id}$.
\end{proof}

Now, let $\sigma$ be a linear $n$-simplex of $\R^n$ and $v$ a vertex of $\sigma$. Consider the stellar subdivision $\mathcal{S}\sigma$ of $\sigma$ and take a closed ball $B$ in $\R^n$ centered at $v$ that does not contain the barycenter $b$ of $\sigma$. Denote by $\pi : B \cap \sigma \to \partial \sigma$ the projection onto $B \cap \partial \sigma$ parallel to $b-v$. 

\begin{lemma}\label{inter2} 
Assume that $\varphi : B \cap \sigma \to \varphi(B \cap \sigma) \subset \sigma$ is a homeomorphism which is smooth on the traces in $B$ of the simplices of $\mathcal{S}\sigma$ and enjoys the following properties~:
\begin{itemize}
\item $\varphi$ fixes the points of $B \cap \partial \sigma$, 
\item $\varphi$ is parallel to the edge $b-v$, in the sense that $\pi(\varphi(x))=\pi(x)$ for any $x \in B \cap \sigma$,
\item $\varphi$ is the identity near $\partial B \cap \partial \sigma$.
\end{itemize}
Then there exists a PS homeomorphism $\overline\varphi$ of $\mathcal{S}\sigma$ which coincides with $\varphi$ near $v$ and with the identity outside $B \cap \sigma$. If $\varphi$ depends smoothly on an additional parameter $\lambda$ that varies in a smooth manifold $\Lambda$, then so does~$\overline\varphi$.
\end{lemma}

\begin{proof}
Observe that $\varphi$ can be seen as a parametric family of maps 
$$
\varphi_\lambda : [0,R_\lambda]\to \R^+,
$$
with $\varphi_\lambda(0) = 0$, where the parameter $\lambda$ varies in $B \cap \partial \sigma$ and $R_\lambda$ denotes the length of the segment $\pi^{-1}(\lambda) \cap B$. Note that 
$$B \cap \partial \sigma = \bigcup_{\tau \in \sigma^{(n-1)}} B \cap \tau,$$
and $\varphi_\lambda$ depends smoothly on $\lambda \in B \cap \tau$, for any $\tau \in \sigma^{(n-1)}$. It follows from Lemma \ref{inter} that there exists a family of diffeomorphisms $\overline\varphi_\lambda$ which coincide with $\varphi_\lambda$ near 0, with the identity near $R_\lambda$, and depends smoothly on $\lambda$ in $B \cap \tau$ for any $\tau \in \sigma^{(n-1)}$. When $\lambda$ approaches $\partial B \cap \partial \sigma$, the maps $\varphi_\lambda$ and $\overline\varphi_\lambda$ are the identity. Consequently, the piecewise smooth family $\{\overline\varphi_\lambda\}_{\lambda\in B \cap \partial \sigma}$ corresponds to a homeomorphism $\overline{\varphi}$ of $B \cap \sigma$, smooth on the traces of the simplices of $\mathcal{S}\sigma$ in $B$ and which satisfies the other desired properties. 
\end{proof}

\section{Uniqueness for triangulations of volume forms}

Let us assume that we have two triangulations $g : |K| \to M$ and $h : |L| \to M$ of a smooth volume form $\Omega$ on a $n$-dimensional manifold $M$. We know by Whitehead's Theorem that there exists a PL equivalence $\varphi : |{K}| \to |{L}|$. Of course $\varphi$ does not necessarily satisfy $\varphi^*\Omega_L = \Omega_K$, for $\Omega_K = g^*\Omega$ and $\Omega_L = h^*\Omega$. (The map $\varphi$ is, in fact, a composition $(g')^{-1} \circ h'$ for maps $g'$ and $h'$ approximating $g$ and $h$ respectively). Nevertheless $\varphi^*\Omega_L$ is another piecewise constant volume form on $K$ and its total volume agrees with that of $\Omega_K$. We explain below how to construct a PL homeomorphism $\psi : K_1 \to K_2$ between two subdivisions of $K$ for which $\psi^*(\varphi^*\Omega_L) = \Omega_K$. This can be considered as a PL version of Moser's argument for volume forms. 

\begin{thm}\label{PLMoser} Any two PC volume forms $\Om_1$ and $\Om_2$ on a closed oriented PL manifold $P$, with identical total volume, are PL equivalent.
\end{thm}

Here is the idea of the proof. Let $n$ denote the dimension of $K$. Let us also consider the difference of volume cocycles 
$$\D_{\Om_1, \Om_2} = \D : K^{(n)} \to \R : \sigma \mapsto \int_\sigma (\Om_2 - \Om_1) = V_{\Om_2}(\sigma) - V_{\Om_1}(\sigma).$$ 
The idea is to remove the "excess of volume" from simplices $\sigma$ for which $\D(\sigma) >0$ by transfering it into simplices with a "deficit of volume" (a negative $\D$). Here \emph{transfering} really means pulling-back $\Om_2$ via a PL homeomorphism between two (perhaps distinct) subdivisions of $K$. This is done by choosing a path of adjacent simplices between a simplex $\sigma_{\rm max}$ with maximal $\D$ to another one, $\sigma_{\rm min}$, with minimal $\D$ and then transfering the excedent of volume from $\sigma_{\rm max}$ to the next simplex and repeating until reaching $\sigma_{\rm min}$, without changing the volume of intermediary simplices. In the end, the simplex $\sigma_{\rm max}$ has same $\Om_1$-volume and $\Om_2$-volume and $\sigma_{\rm min}$ has reduced its deficit, perhaps to the point of having an excess but the later is lower than the original one of $\sigma_{\rm max}$. We may repeat this process until having suppressed all excess in volume. This yields a sequence $\psi_1, ..., \psi_k$ of PL self-equivalences of $K$ that we may compose to obtain the desired $\psi$. (It is indeed a standard fact that the collection of PL manifolds, together with PL maps form a category.) \\

We will use the following subdivision of a pair of adjacent simplices $\sigma$ and $\tau$. Add to the collection of vertices $\sigma^{(0)} \cup \tau^{(0)}$ a vertex $v$ inside $\sigma$ and a vertex $u$ inside $\sigma \cap \tau$ (not necessarily barycenters) and consider the stellar subdivision $\ES (\sigma \cup \tau)$ centered at $u$. It induces a \emph{conical} subdivision of $\tau$ obtained by considering the simplices generated by all possible collection of vertices in $\tau^{(0)} \cup \{u\}$, except $\sigma \cap \tau$. For $\sigma$, we first consider its stellar subdivision $\ES \sigma$ centered at $v$ and replace the $n$-simplex $<v, \sigma \cap \tau>$ by the conical subdivision associated to $\ES (\sigma \cap \tau)$. Denote that new simplicial complex by $\mathcal S_{v, u}(\sigma, \tau)$ and the induced subdivision of $K$ by $\mathcal S_{v, w}(K)$.


\begin{lemma}\label{transfer} [Transfering volume between adjacent simplices.] Let $\sigma, \tau \in K^{(n)}$ be two adjacent $n$-simplices of $K$ and let $0 < A < \int_\tau\Om_2$. Then for some choice of $v \in {\rm Int} \sigma$, $w \in {\rm Int} \tau$, $u_\sigma, u_\tau \in {\rm Int}(\sigma \cup \tau)$, the PL map 
$$\psi : {\mathcal S}_{v, u_\tau}(\sigma, \tau) \to {\mathcal S}_{w, u_\sigma}(\tau, \sigma)$$ 
which maps $v$ to $u_\sigma$, $u_\tau$ to $w$ and fixes the other vertices satisfies the following conditions~:
\begin{enumerate} 
\item[(i)] $\psi = \rm Id$ on $|\partial (\sigma \cup \tau)|$, 
\item[(ii)] $(\psi^*\Om_2)|_{\tau}$ and $(\psi^*\Om_2)|_{\sigma}$ are constant,
\item[(iii)] $\int_\tau\psi^*\Om_2 = A$ (and $\int_\sigma\psi^*\Om_2 = V_{\Om_2}(\sigma) + V_{\Om_2}(\tau) - A$).
\end{enumerate}
\end{lemma}


\begin{proof} Let us denote by $\theta$ the $(n-1)$-simplex $\sigma\cap \tau$. We introduce the notation $\a_\sigma$ for the affine plane spanned by a simplex $\sigma \subset \R^N$.\\

We first discuss how to obtain $A$ as total volume of $(\psi^*\Om_2)|_{\tau}$. The PL map $\psi$ removes the $\Om_2$-volume carried by $<w, \theta>$ from $\sigma$ and transfer it to $<v, \theta>$. Adjusting $w$ allows to choose the size of $V_{\Om_2}(<w, \theta>)$. The relation to be satisfied is 
$$V_{\Om_2}(<w, \theta>) = V_{\Om_2}(\sigma) - A.$$ 
It determines a $(n-1)$-simplex $\theta_A^\tau$ in $\tau$, parallel to $\theta$, 
along which to choose $w$.\\

Concerning the condition for $(\psi^*\Om_2)|_{\tau}$ to be constant, observe that, with respect to the subdivision ${\mathcal S}_{v, u_\tau}(\sigma, \tau)$, the simplex $\tau$ is divided into $n$ $n$-simplices~: the $<u_\tau, \nu_i>$'s, where $\nu_1, ..., \nu_n$ are the $(n-1)$-faces of $\tau$ distinct from $\theta$. Set  $\nu^*_i = <u_\tau, \nu_i>$. Of course, $\psi(\nu_i^*) = <w, \nu_i>$. To have $(\psi^*\Om_2)|_{\tau}$ constant, the ratios $V_{\Om_2}(\psi(\nu_i^*)) / V_{\Om_2}(\nu_i^*)$ must all be equal (to $A/V_{\Om_2}(\sigma)$). \\

To simplify the discussion, observe that we may even choose $u_\tau$ and $w$ so that all $V_{\Om_2}(\nu_i^*)$ are equal and idem for the $V_{\Om_2}(\psi(\nu_i^*))$'s. Indeed, for $u_\tau$, this condition amounts to a system of $n-1$ independant equations $V_{\Om_2}(<u_\tau, \nu_i>) = V_{\Om_2}(\tau)/n$, $i = 1, ..., n-1$ on $u_\tau$ (since the relation $V_{\Om_2}(<u_\tau, \nu_n>) = V_{\Om_2}(\tau)/n$ is redundant), each of which determines a hyperplane in $\a_{\tau}$ parallel to $\a_{\nu_i}$ and situated between $\a_{\nu_i}$ and its parallel copy passing through the vertex opposite to $\nu_i$ in $\tau$. They all intersect into a unique point in $\theta$. The discussion for $w$ is almost identical. Indeed, the relations to be satisfied by $w$ are~: $V_{\Om_2}(<w, \nu_i>) = A/n$ for $i = 1, ..., n-1$, which determine a unique point in $\theta_A^\tau$. \\

The next task is to ensure that $\psi^*\Om_2$ is constant on $\sigma$. Here too, it is possible to choose $v$ and $u_\sigma$ so that the $\psi^*(\Om_2)$-volume of $<v, \theta>$ agrees with that of the simplices $<v, \mu_j>$ for $\mu_1, ..., \mu_n$, the collection of elements in $\sigma^{(n-1)} \setminus \{\theta\}$. This is very similar to the previous discussion and amounts to choosing $v$ in the intersection of the $(n-1)$-simplex $\theta_A^\sigma \subset \sigma$ parallel to $\theta$ determined by the relation 
$$V_{\Om_2}(<w, \theta>) = \frac{1}{n} V_{\Om_2}(\sigma)$$
with the $n-1$ independent affine hyperplanes 
$$V_{\Om_2}(<u_\sigma, \mu_i>) = V_{\Om_2}(\tau) - A.$$ 

There remain one thing to do~: to ensure that $\psi^*\Om_2$ is constant on $<v, \theta>$. The latter simplex is cut into $n$ $n$-simplices~: the $<v, u_\tau, \gamma_k>$, with $\{\gamma_1, ..., \gamma_n\} = \theta^{(n-2)}$. Of course $V_{\psi^*\Om_2}(<v, u_\tau, \gamma_k>) = V_{\Om_2}(\psi(<v, u_\tau, \gamma_k>)) = V_{\Om_2}(<u_\sigma, w, \gamma_k>)$. Observe that our choice for $u_\sigma$ (\rp $u_\tau$) guarantees that the $\Om_2$-volume of the various $<u_\sigma, w, \gamma_k>$'s (\rp $<v, u_\tau, \gamma_k>$'s) are identical.
\end{proof}

{\it Proof of \tref{PLMoser}} There are many ways to proceed. We describe one that requires few words but a lot of steps. There are more ``economical procedures".\\

Observe that because $K$ is finite, the difference cocycle $\D$ is necessarily bounded and let $\sigma_{\rm max}$ denote a $n$-simplex in $K$ for which $\D(\sigma)$ is maximal and another one, $\sigma_{\rm min}$, for which $\D(\sigma_{\rm min})$ is minimal. Then find a path of adjacent $n$-simplices starting at $\sigma_{\rm max}$ and ending at $\sigma_{\rm min}$, say $\sigma_0 = \sigma_{\rm max}, \sigma_1, ..., \sigma_k = \sigma_{\rm min}$ and use \pref{transfer} to find $v \in \sigma_1$, $w \in \sigma_0$, $u_{\sigma_0}, u_{\sigma_1} \in \sigma_0 \cap \sigma_1$ such that $\psi_0 : {\mathcal S}_{v, u_{\sigma_1}}(K) \to {\mathcal S}_{w, u_{\sigma_0}}(K)$ achieves $V_{{\psi_0}^*\Om_2}(\sigma_0) = V_{\Om_1}(\sigma_0)$. \\

We redo this transfer with $\sigma_0$ (\rp $\sigma_1$) replaced by $\sigma_1$ (\rp $\sigma_2$), constructing $\psi_1$ such that $V_{\psi_1^* (\psi_0^*\Om_2)}(\sigma_1) = V_{\Om_2}(\sigma_1)$. Now the excess of $\Om_2$-volume of $\sigma_0$ has been transferred, via $\psi_0 \circ \psi_1$, to $\sigma_2$, without modifying the volume of the intermediary simplex $\sigma_1$. We repeat this step until having transferred $\D(\sigma_0)$ of the initial volume of $\sigma_0$ to $\sigma_k$ via a composition $\Psi = \psi_0 \circ ...\circ \psi_k$ of PL self-equivalences of $K$.\\

Observe that $\Om_2' = \Psi^*\Om_2$ is piecewise constant on $K$, the $\Om'_2$-volume of $\sigma_{\rm max}$ agrees with its $\Om_1$-volume and the $\Om'_2$-volume of $\sigma_{\rm min}$ is $V_{\Om_2}(\sigma_{\rm min}) + \D(\sigma_{\rm max})$, which means that $\D_{\Om'_2, \Om_1}(\sigma_{\rm min}) < \D(\sigma_{\rm max})$. We now replace $\Om_2$ by $\Om'_2$ and repeat from the beginning. This will not damage what has been achieved so far, since even if our new path includes $\sigma_{\rm max}$ as an intermediary simplex, its new volume will remain identical.
Since we do not increase the total number of simplices, after finitely many steps (at most the cardinality of $K^{(n)}$), we have equally distributed the volume of $\Om_2$ with respect~to~$\Om_1$.\hfill $\square$

\section{Several Consequences}

Let us now describe a few consequences of the previous results.

\begin{prop}\label{existence}
Consider a smooth manifold $M$ equipped with a smooth volume form $\Omega$. Let $h : |K|\to M$ be a smooth triangulation and let $\Omega_1 = {\mathcal C}_K(h^* \Omega)$. Then there exists a subdivision $\overline{K}$ of $K$ and a triangulation $\overline h : |\overline K| \to M$ such that $\overline{h}^*\Omega = \Omega_1$.
\end{prop}

\begin{proof}
By Proposition \ref{conclusion1} and Proposition \ref{conclusion2} applied successively for $k = 1, ..., n-1$ , there exists a subdivision $\ol{K}$ of $K$ (the result of $n+1$ successive stellar subdivisions) and a smooth triangulation $h_o : |\overline K| \to M$ for which $h_o^*\Omega$ and $\Omega_1$ have the same volume cocycle on $\overline{K}$ and coincide on a neighborhood of $|\overline{K}^{(n-1)}|$. Now Proposition \ref{mo} can be applied in each $n$-simplex of $\overline{K}$ to provide us with a piecewise smooth homeomorphism $\varphi$ of $\overline{K}$ supported away from the skeleton of $\overline{K}$ and such that $\varphi^*(h_o^*\Omega)=\Omega_1$. The triangulation $\overline{h} = h_o \circ \varphi : |\overline{K}|\to M$ satisfies $\overline{h}^*\Omega = \Omega_1$.
\end{proof}

Consider two volume forms on a closed manifold $M$, one is smooth and the other one is piecewise smooth with respect to some triangulation of $M$. Assume only that they have the same total volume. Then it is possible to construct an equivalence between them.

\begin{cor}[Moser PC -- $C^{\infty}$]\label{smoothmo}
Let $(M, \Omega)$ be a closed manifold equipped with a smooth volume form. Consider a triangulation $h : |K| \to M$ and a piecewise constant volume form $\Omega_1$ on $K$. Assume that 
$$\int_M \Omega = \int_K \Omega_1.$$
Then there exists a subdivision $\overline{K}$ of $K$ and a triangulation $\overline{h} : |\overline{K}| \to M$ such that $\overline{h}^*\Omega = \Omega_1$.
\end{cor}

\begin{proof}
Observe first that we can choose a positive smooth function $f$ on $M$ such that
$$
\int_{h(\sigma)} f \Omega = \int_{\sigma} \Omega_1
$$
for all $\sigma \in K^{(n)}$. The smooth volume forms $\Omega$ and $f\Omega$ lie in the same cohomology class since they have the same total volume.
Thus the standard Moser argument implies that there exists a diffeomorphism $\rho$ of $M$ such that $\rho^*\Omega = f\Omega$. Now the volume cocycles of $f\Omega$ and $\Omega_1$ coincide so that Proposition~\ref{existence} yields a subdivision $\overline{K}$ of $K$ and a triangulation $h_o : |\overline{K}| \to M$ such that $h_o^*(f\Omega) = \Omega_1$. The composition $\ol{h} = \rho \circ h_o$ is a triangulation of $M$ for which
$$
\ol{h}^* \Omega = h_o^*(\rho^*\Omega) = h_o^*(f\Omega) = \Omega_1.$$
\end{proof}

\begin{cor}\label{smPC}[Smoothing of a piecewise constant volume form]
Let $P$ be a closed PL manifold equipped with a piecewise constant volume form $\Omega_P$. Assume that $P$ admits a smoothing, that is, there exists a smooth manifold $M$, and a triangulation $h : P \to M$. Then there exists a smooth volume form $\Omega$ on $M$, and a triangulation $\overline{h} : P \to M$ such that $\overline{h}^*\Omega = \Omega_P$.
\end{cor}

\begin{proof}
Consider a smooth volume form $\Omega$ on $M$ such that $(M, \Omega)$ and $(P, \Omega_P)$ have the same total volume (notice that $P$ and $M$ are necessarily orientable). We may even assume that $\Omega$ and $\Omega_P$ have the same volume cocycle and therefore invoke Proposition \ref{existence} to conclude.
\end{proof}

\begin{rem} The techniques described in the previous section allow also to triangulated a PS volume form $\Omega_P$ on a PL manifold $P$, that is find a PS map $\varphi : P \to P$ such that $\varphi^*\Omega_P$ is piecewise constant, so that \cref{smoothmo} and \cref{smPC} can be stated for piecewise smooth volume forms as well.
\end{rem}

\section{Symplectic jiggling lemma}\label{section jiggling}

We state and prove here a symplectic version of Thurston's jiggling lemma.

\subsection{Discussion of the statement}

Let us first recall the precise statement of Thurston's jiggling Lemma (see \cite{thurstonjig}). Given a $C^0$ plane field $\tau$ on a manifold $M$, a Whitehead triangulation $h : |K| \to M$ and a compact subset $V$ of $M$, there exists natural number $\ell$ and a slight deformation $h'$ of the triangulation $h : \mathcal{S}_\ell(K) \to M$ which is in general position with respect to $\tau$ on $V$, where \textsl{being in general position} means not only that all the simplices of $\mathcal{S}_\ell(K)$ are transverse to $\tau$, but also that for any point $x$ in the interior of $\sigma \in \mathcal{S}_\ell(K)$, the plane $\tau_x$ is transverse to all the faces of $\sigma$. 

\medskip

In the presence of a symplectic form rather than a distribution, the notion of \textsl{general position} and the statement of the \textsl{jiggling Lemma} are adapted as follows.

\begin{df}\label{genpos}
A simplicial complex $K$ is said to be \textsl{in general position} with respect to a symplectic Whitney $2$-form $\omega$ on $K$ if for any $x$ in a simplex $\sigma \in K$ and for any face $\tau$ of $\sigma$, $\omega_x$ has maximal rank along $\tau$. If $M$ is a manifold equipped with a smooth symplectic form $\omega_o$, a Whitehead triangulation $h : |K| \to M$ is \textsl{in general position with respect to $\omega_o$} if $K$ is in general position with respect to $h^*\omega_o$.
\end{df}

\begin{lemma}[Symplectic jiggling lemma]\label{sympjig1} 
Any closed symplectic manifold $(M, \om)$ admits Whitehead triangulations that are in general position with respect to $\om$. 
\end{lemma}

To begin, the manifold $M$ is smoothly embedded in $\R^k$ for a certain $k > 2n$. If $\pi : N \subseteq \R^k \to M$ is a tubular neighborhood of $M$, there exists a PL manifold $K$ contained in $N$ and transverse to $\pi$ such that $\pi|_{|K|} : |K| \to M$ is a Whitehead triangulation of $M$ (cf. \cite{git}, Chapter IV, B). Notice that we do not require that $\omega$ is induced from the standard (or any other) symplectic structure on $\R^k$. Set $\Omega = \pi^*\om$.\\

We now say a few words about the notion of jiggling. We suppose here and later that $K$ has perhaps been replaced by $\mathcal S_n(K)$ and that it contains the model simplices for its crystalline subdivisions. Fix a $\delta >0$ such that for each simplex $\sigma=\langle v_{i_0}, \ldots, v_{i_n} \rangle$ in $K$, any choice of points $v_{i_0}', \ldots, v_{i_n}'$, with $v_{i_j}'$ is in the ball of radius $\delta$ centered at $v_{i_j}$, determines a non-degenerate $n$-simplex. We call such a simplex a \textsl{$\delta$-jiggling of $\sigma$}. Given a $n$-simplex $\mu$ of $\mathcal{S}_\ell(K)$, a $\delta$-jiggling of the corresponding model simplex induces a $\frac{\delta}{\ell}$-jiggling of $\mu$, and vice-versa.
Therefore, if $x_0, \ldots, x_m$ are the vertices of $\mathcal{S}_\ell(K)$, each choice of points
$$
x_0' \in B\left(x_0, \frac{\delta}{\ell}\right), \ldots, x_m' \in B\left(x_m, \frac{\delta}{\ell}\right)
$$ 
determines a new simplicial complex $\mathcal{S}'_\ell(K)$, called a \textsl{jiggling} or $\frac{\delta}{\ell}$-\textsl{jiggling} of $\mathcal{S}_\ell(K)$ which is PL homeomorphic to $\mathcal{S}_\ell(K)$. The terminology \textsl{closed jiggling} will be used to indicate that the vertices move in closed balls. In addition, the positive number $\delta$ can be chosen sufficiently small for any $\frac{\delta}{\ell}$-jiggling of $\mathcal{S}_\ell(K)$ to remain transverse to $\pi$. Such a jiggling is called hereafter an \textsl{admissible jiggling}. \\

Then \lref{sympjig1} is a consequence of the next result.

\begin{lemma}\label{sympjig}
There exists a crystalline subdivision $\mathcal{S}_\ell(K)$ of $K$ and an admissible jiggling $\mathcal{S}'_\ell(K)$ of $\mathcal{S}_\ell(K)$ such that $\pi : |\mathcal{S}'_\ell(K)| \to M$ is in general position with respect to~$\omega$.
\end{lemma}

As an oversimplified example, consider a linear simplex $\sigma$ of dimension two in the standard symplectic space $(\R^{2n},\omega_0)$. The restriction of $\omega_0$ may be degenerate along $\sigma$, but it is easy to see that there exists arbitrary small jigglings of $\sigma$ along which $\omega_0$ has maximal rank. If $\sigma = \langle v_0, v_1, v_2 \rangle$, it is indeed sufficient to move $v_2$ away from the affine hyperplane
$$
S=\{x \in \R^{2n} : \omega_0(x-v_0,v_1-v_0)=0\}.
$$

More generally, fix a $q$-simplex $\sigma \in K$ and observe that, except for $q=1, 2n-1$ or $2n$, the form $\pi^*\omega$ does not necessarily have maximal rank along $\sigma$. Moreover, since the form $\pi^*\omega$ varies within $\sigma$, it is less trivial than in the previous example that $\sigma$ admits a jiggling $\sigma'$ such that $(\pi^*\omega)_x$ has maximal rank along $\sigma'$ for any $x \in \sigma'$. This suggests that we should first subdivide $\sigma$ into sufficiently small simplices to control the variation of the smooth form $\pi^*\omega$ within a simplex. On the other hand, it is necessary to take into account the fact that small simplices have smaller amplitude jigglings. The difficulty is to find a balance between those two facts.

\subsection{Crystalline subdivision of a simplicial complex}\label{crystalline}

We recall here the notion crystalline subdivision, as well as some of its properties that are necessary for our purpose.

\medskip

Consider a finite simplicial complex $K \subseteq \R^k$, and choose an ordering of the vertices $v_0, v_1, \ldots, v_p$. Any $n$-simplex $\sigma$ of $K$ is uniquely generated by vertices $v_{i_0}, v_{i_1}, \ldots, v_{i_n}$, with $0 \leq i_0 < \ldots < i_n \leq p$. Observe that the linear map $f$ defined by the conditions $f(v_{i_0})=(1,1,\ldots,1)$, $f(v_{i_1})=(0,1,\ldots,1)$, $\ldots$, $f(v_{i_n})=(0,0, \dots, 0)$, sends $\sigma$ into the $n$-cube
$$
C=\{(x_1, \ldots, x_n) : 0 \leq x_i \leq 1 ~~\mathrm{for ~any~~} i\in\{1,\ldots,n\}\} \subseteq \R^n.
$$ 
The later can be subdivided into $n!$ simplices~: for any permutation $\tau$ of $n$ elements, consider the $n$-simplex defined by
$$
s_\tau = \{(x_1, \ldots, x_n) : 0 \leq x_{\tau(1)} \leq x_{\tau(2)} \leq \ldots \leq x_{\tau(n)} \leq 1\}.
$$
Notice that $f(\sigma)=s_{\mathrm{Id}}$.

Now, for any $\ell \in \N_0$, one may subdivide $C$ into $\ell^n$ isometric cubes and decompose each of them into $n$-simplices as above. This induces a subdivision of $\sigma$, called its \textsl{crystalline subdivision of order $\ell$}.  Moreover, the crystalline subdivisions of all the $n$-simplices of $K$ fit together nicely since the induced subdivision on each face is the crystalline subdivision of the face and yield the \textsl{crystalline subdivision} of order $\ell$ of $K$, denoted $\mathcal{S}_\ell(K)$. \\

\begin{figure}[htb]
  \centering
  \def\svgwidth{180pt}
  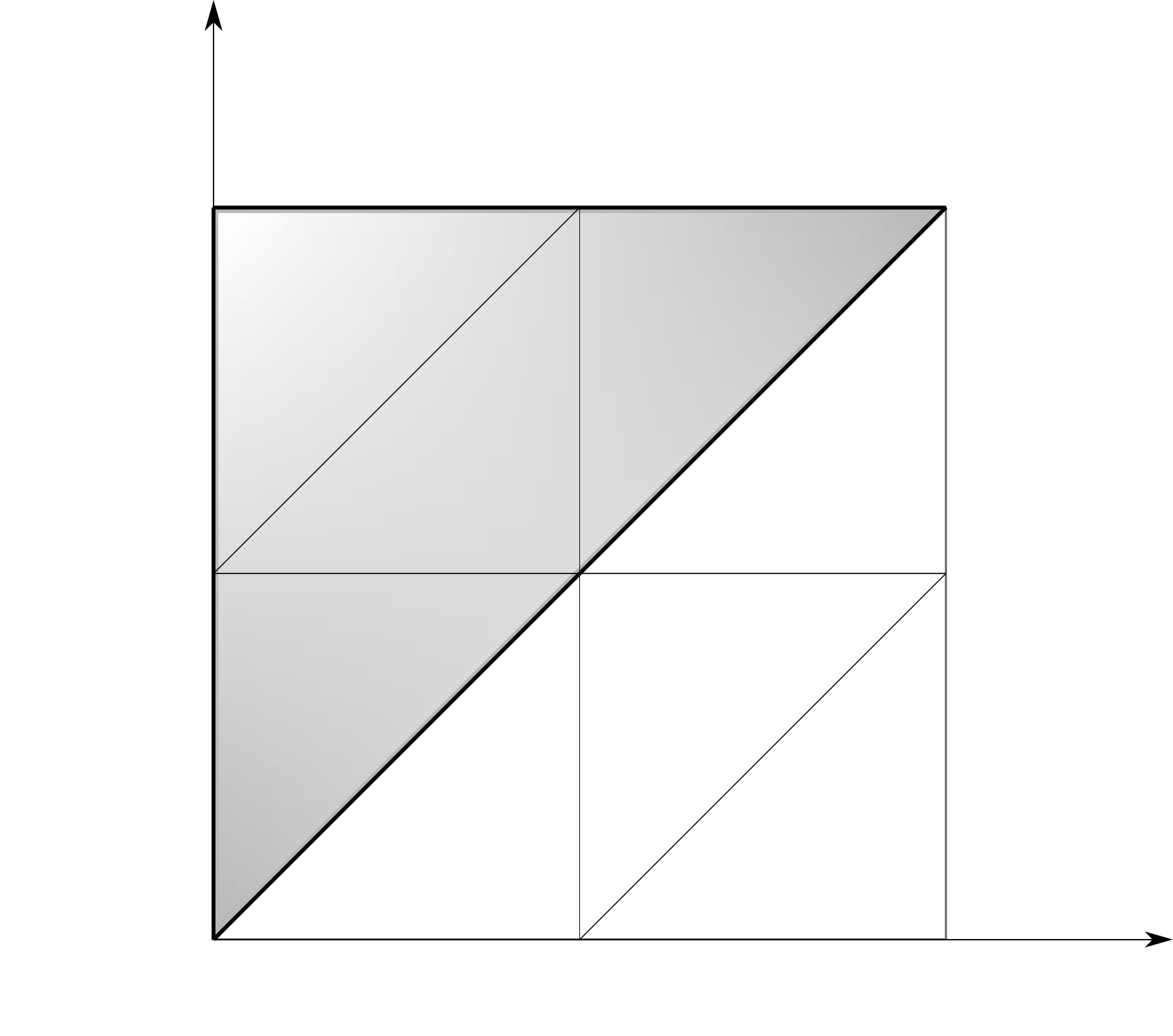
\caption{Crystalline subdivision of order two of a $2$-simplex (in grey).}\label{cryst}
\end{figure}

As illustrated on Figure \ref{cryst}, the crystalline subdivision of a 2-simplex $\sigma$ contains two types of simplices, up to translations, which correspond to the two simplices of the standard subdivision of the $2$-cube. Therefore, given a complex $K$ of dimension two, the complex $\mathcal{S}_2(K)$ has the property that if we consider any of its crystalline subdivisions, the new simplices of dimension two are all obtained from those of $\mathcal{S}_2(K)$ by a contraction followed by a translation. For a complex $K$ of dimension $n$, it is the crystalline subdivision of order $n$ or higher that enjoys this property.

\begin{lemma}\label{crystalline property}
Let $K$ be a combinatorial manifold of dimension $n$. For any $\ell \in \N_0$, the crystalline subdivision $\mathcal{S}_\ell(K)$ has the property that all its simplices are obtained from the simplices of $\mathcal{S}_n(K)$ by a homothety of factor $\frac{\ell}{n}$ followed by a translation.
\end{lemma}

\begin{proof}
Fix a $n$-simplex $\sigma$ of $K$. The linear map $f$ defined previously sends $\sigma$ onto the $n$-simplex
$$
f(\sigma)=\{(x_1,\ldots,x_n) \in \R^n : 0 \leq x_1 \leq \ldots \leq x_n \leq 1 \} \subseteq C,
$$
which contains one of the cubes of the subdivision of $C$ into $n^n$ cubes, namely~:  
$$
\left[0,\frac{1}{n}\right] \times \left[\frac{1}{n},\frac{2}{n}\right] \times \ldots \times \left[ \frac{n-1}{n},1\right].
$$
Consequently, the crystalline subdivision of order $n$ of $\sigma$ contains all the simplices corresponding to the standard subdivision of this little cube. 
\end{proof}

Therefore, replacing $K$ by $\mathcal{S}_n(K)$ if necessary, we may assume the following properties~:

\begin{enumerate}
\item[$\bullet$] For any $\ell > 1$, each simplex $\sigma_o$ of $\mathcal{S}_\ell(K)$ is obtained from a simplex $\sigma$ of $K$ by a contraction of factor $\frac{1}{\ell}$ followed by a translation. We may even choose $\sigma$ so that $\sigma_o \subset \sigma$. The simplices of $K$ are called \textsl{model simplices}. 
\item[$\bullet$] The crystalline subdivision of a simplex does not increase the number of simplices incident to an initial vertex. It follows that there is a uniform bound $A$, independent on $\ell$, on the number of simplices in the link of a vertex $v$ of $\mathcal{S}_\ell(K)$. \\
\end{enumerate}


\subsection{The proof of the jiggling lemma}

We will now give the proof of Lemma \ref{sympjig}. It consists of an induction on the dimension of the skeleton of $K$ which is in general position with respect to $\pi^*\om$. Set $\Omega=\pi^*\omega$. The Lebesgue measure of a measurable subset $S$ of $\R^k$ is denoted hereafter by $m(S)$ and the closed $\delta$-neighborhood of the body of a simplicial complex $L$ is denoted by $N_\delta(L)$. \\

For any $q \in \{1, \ldots, 2n\}$ and any $x \in N$, we define the following subset of the Grassmannian $G_q(\R^k)$ of vector subspaces of dimension $q$ in $\R^k$ 
$$
\mathcal{L}_x^q=\{L \in G_q(\R^k) : \mathrm{the~ rank ~of~} (\pi^*\omega)_x \mathrm{~is~ not~ maximal~ along~} L\}.
$$
Observe that if $q$ is even, the rank of $(\pi^*\omega)_x$ is not maximal along $L \in G_q(\R^k)$ if and only if given a basis $\{e_1,\ldots,e_q\}$ of $L$, the skew-symmetric matrix $\left((\pi^*\omega)_x(e_i,e_j)\right)_{ij}$ has vanishing determinant.
For $q$ odd, the rank of the matrix is not maximal if all the minors of size $q-1$ have vanishing determinant. Thus, in both cases, $\mathcal{L}_x^q$ is a compact subset of $G_q(\R^k)$. \\

Observe that to prove the symplectic jiggling lemma, it suffices to construct an admissible jiggling $\mathcal{S}'_\ell(K)$ of a certain subdivision $\mathcal{S}_\ell(K)$ of $K$, sufficiently small for $\pi : |\mathcal{S}'_\ell(K)| \to M$ to remain a triangulation of $M$ and such that each $q$-simplex $\sigma$ of $\mathcal{S}'_\ell(K)$ generates a plane which is not in any of the $\mathcal{L}_x^q$'s for $x$ running in $\sigma$. \\

To measure the variation of the $\mathcal{L}_x^q$'s within a simplex, we fix a metric $d$ on the Grassmannian $G_q(\R^k)$ and consider the induced \textsl{Hausdorff distance} 
$$
d_{\mathcal{H}}(K_1, K_2) = \max \left\{ \sup_{x_1 \in K_1} d(x_1,K_2), \sup_{x_2 \in K_2} d(x_2,K_1) \right\}
$$
between compact subsets $K_1, K_2$ of $G_q(\R^k)$. Notice in particular that $d_{\mathcal{H}}(K_1,K_2) < \varepsilon$ if and only if $K_1 \subseteq B(K_2,\varepsilon)$ and $K_2 \subseteq B(K_1,\varepsilon)$ and that it makes sense to talk about the Hausdorff distance between $\mathcal{L}_x^q$ and $\mathcal{L}_y^q$ for $x, y$ in $N$. The following Lemma is proven in \cite{D19} (Lemma 4.2.3 p~52).

\begin{lemma}\label{conti}
For any $q \in \{1,\ldots, 2n\}$, the map $
x \mapsto \mathcal{L}_x^q$
is continuous with respect to $d_{\mathcal{H}}$.
\end{lemma}

As a consequence, the Hausdorff distance between the various $\mathcal{L}_x^q$'s for $x$ varying in the star of a simplex $\sigma \in \mathcal S_\ell(K)$ can be made arbitrarily small by increasing $\ell$. 

\begin{prop}\label{jiggling}
Let us assume that, for a certain $q \in \{2, 4, \ldots, 2n-2\}$, our triangulation $\pi|_{|K|} : |K| \to M$ has the property that all the $j$-simplices of $K$ are in general position with respect to $\pi^*\omega$, for $j = 1, ..., q-1$. Then there exists a crystalline subdivision $\mathcal{S}_\ell(K)$ of $K$ and an admissible jiggling $\mathcal{S}'_\ell(K)$ of $\mathcal{S}_\ell(K)$ such that all the $j$-simplices of $\mathcal{S}'_\ell(K)$ are in general position with respect to $\Omega$ for $j = 1, ..., q+1$. 
\end{prop}

\begin{rem}\label{delta}
If, for $j = 1, ..., q-1$, the $j$-simplices of $K$ are in general position with respect to $\Omega$, that is, if, for any $j$-simplex $\tau$ with $j = 2, ..., q-2$, the form $\Omega_u$ is non-degenerate on $\tau$ for all $u$ in $\st_K(\tau)$, then the number $\delta$ can be chosen small enough for the same property to hold for all closed $\delta$-jigglings of $\tau$, for all $u$ in the closed $\delta$-neighborhood of $\st_K(\tau)$ and for all $j$-simplex $\tau$ with $j = 2, ..., q-2$. 

Notice that, for such a $\delta$ and for any crystalline subdivision $\mathcal{S}_\ell(K)$ of $K$, the form $\Omega_u$ is non-degenerate on each $\frac{\delta}{\ell}$-jiggling of a $j$-simplex $\tau$ with $j = 2, ..., q-2$ of $\mathcal{S}_\ell(K)$ for all $u$ in $N_\delta(\st_{\mathcal{S}_\ell(K)}(\tau))$ because the faces created by a crystalline subdivision are necessarily parallel to the original ones and the faces of a $\frac{\delta}{\ell}$-jiggling of $\mathcal{S}_\ell(K)$ are parallel to those of the corresponding $\delta$-jiggling of $K$ (not to forget that the star of a simplex in $\mathcal{S}_\ell(K)$ is contained in the star of its model simplex in $K$ if the later is chosen to contain the former). 

Therefore, for the conclusion of the lemma to hold, it suffices to restrict ourself to such sufficiently small $\delta$'s and to achieve that for each $q$-simplex $\sigma$ of the $\frac{\delta}{\ell}$-jiggled complex $\mathcal{S}'_\ell(K)$, and for each point $u\in \st_K(\sigma)$, the form $\Omega_u$ is non-degenerate on $\sigma$, or, equivalently, the $q$-plane $L_\sigma$ is not contained in $\mathcal{L}_u^q$. 
\end{rem}

Introduce, for each model $q$-simplex $\sigma$, for each vertex $v$ of $\sigma$, for each $\delta$-jiggling $\tau'$ of the $(q-1)$-simplex $\tau$ opposite to $v$ in $\sigma$, for each $u$ in the $\delta$-neighborhood of $\st_K(v)$, and for each $\varepsilon>0$, the ``sector''
$$
S_\varepsilon(\sigma, v, \tau', u) = \Bigl\{ x \in \R^k : L_{\langle x,\tau'\rangle} \in G_q(\R^k) \mbox{ and } d_\mathcal{H}(L_{\langle x,\tau' \rangle}, \mathcal{L}^q_u) < \varepsilon \Bigr\},
$$
and $S_0(\sigma,v,\tau',u) = \{x \in \R^k \;|\; L_{\langle x, \tau' \rangle} \in \mathcal{L}_u^q\}$.
We explain below how to choose a sufficiently small positive number $\varepsilon$ for the ball $B(v,\delta)$ to contain points which do not belong to the sector $S_\varepsilon(\sigma, v, \tau',u)$. Such a point $x$ has the property that the jiggled $q$-plane $L_{\langle x,\tau' \rangle}$ is at distance at least $\varepsilon$ from $\mathcal{L}_u^q$, and therefore does not intersect the $\mathcal{L}^q_{u'}$'s with 
\begin{equation}\label{ineq}
d_{\mathcal{H}}(\mathcal{L}_u^q,\mathcal{L}_{u'}^q) < \varepsilon.
\end{equation}
Consequently, $\varepsilon$ will determine the order $\ell$ of the crystalline subdivision of $K$ that is needed to ensure that the inequality (\ref{ineq}) holds within any simplex of ${\mathcal S}_\ell(K)$. 

\begin{lemma}\label{epsilon}
There exists $\varepsilon > 0$ such that 
$$
m\Bigl(S_{\varepsilon_o}(\sigma, v,\tau', u) \cap B(v,\delta)\Bigr) < \frac{m\left(B(v,\delta)\right)}{A},
$$
for any $\varepsilon_o \leq \varepsilon$, any $\sigma$ in $K^{(q)}$, any $v$ in $\sigma^{(0)}$, any closed $\delta$-jiggling $\tau'$ of $\tau$ and any $u$ in $N_\delta(\st_K(v))$, where $A$ denotes the maximal number of simplices in the link of a vertex in $K$ (whence in any crystalline subdivision of $K$).
\end{lemma}

\begin{proof}
The open set $S_\varepsilon(\sigma, v,\tau', u) \cap B(v,\delta)$ has a finite Lebesgue measure which decreases with $\varepsilon$. It follows that when $\varepsilon$ tends to $0$, the measure of $S_\varepsilon \cap B(v,\delta)$ converges to the measure of $S_0 \cap B(v,\delta)$.
Observe that 
\begin{align*}
S_0(\sigma,v,\tau',u) \cap B(v,\delta) &= \{x \in B(v,\delta) : L_{\langle x ,\tau' \rangle} \in \mathcal{L}_u^q\}\\
&= \{ x \in B(v, \delta) : \mathrm{det}\left(\Omega_u|_{L_{\langle x, \tau' \rangle}}\right)=0\}.
\end{align*}
By the choice of $\delta$ (\rref{delta}), we know that $\Omega_u$ is non-degenerate along the $(q-2)$-faces of the jiggled simplex $\tau'$. We can therefore choose a basis $\{e_1,\ldots,e_{q-1}\}$ of $L_{\tau'}$ such that $e_1$ lies in the kernel of $\Omega_u$ restricted to $L_{\tau'}$ (and the plane spanned by $e_2, \ldots, e_{q-1}$ is therefore symplectic). Complete this basis into a basis $\{e_0, e_1, \ldots, e_{q-1}\}$ of $L_{\langle x,\tau' \rangle}$ by specifying $e_0 = x-y_0$, for some vertex $y_0$ of $\tau'$. Then the matrix of $\Omega_u|_{\langle x, \tau' \rangle}$ can be written as

\begin{align*}
\left(
\begin{array}{c|c}
\begin{array}{cc}
0 &\Omega_u(e_0,e_1)\\
-\Omega_u(e_0,e_1) & 0 
\end{array}
&\begin{array}{cccc}
 \Omega_u(e_0,e_2) & \cdots & \Omega_u(e_0,e_{q-1}) \\
  0 & \cdots & 0 
\end{array}\\
\hline
\begin{array}{cc}
\hspace{.6cm} -\Omega_u(e_0,e_2)\hphantom{0000} &0 \hphantom{0000000}\\
\hspace{.6cm} \vdots \hphantom{0000} & \vdots \hphantom{0000000} \\
\hspace{.6cm} -\Omega_u(e_0,e_{q-1})\hphantom{00}&0 \hphantom{0000000}
\end{array}
 & \Huge{M}
\end{array}
\right),
\end{align*}
where $M$ is a skew-symmetric matrix of size $q-2$ with non-vanishing determinant. 
Observe that
$$
\mathrm{det}\left(\Omega_u|_{L_{\langle x,\tau'\rangle}}\right) = \Bigl(\Omega_u(e_0,e_1)\Bigr)^2 \; \mathrm{det}(M),
$$
which says that a point $x \in B(v,\delta)$ belongs to $S_0(\sigma,v,\tau',u)$ if and only if $\Omega_u(x-y_0,e_1)=0$. The later condition describes a an affine hyperplane of $\R^k$, that is a subset whose Lebesgue measure vanishes. Consequently the function $m\left(S_\varepsilon (\sigma, v,\tau', u) \cap B(v,\delta)\right)$ tends to 0 with $\varepsilon$, and since the choices of $\sigma$, $v$, $\tau'$, $u$ are either finite or range over compact sets (see statement of \lref{epsilon}), we can fix an $\varepsilon >0$ such that 
$$
m\Bigl(S_{\varepsilon_o}(\sigma, v,\tau', u) \cap B(v,\delta) \Bigr) < \frac{m\left(B(v,\delta)\right)}{A},
$$
whoever are $\sigma$, $v$, $\tau'$, $u$ and $\varepsilon_o \leq \varepsilon$. 
\end{proof}

Consider any crystalline subdivision $\mathcal{S}_\ell(K)$ of $K$. Let $\sigma_o$ be a $q$-simplex of $\mathcal{S}_\ell(K)$ and $\sigma$ the corresponding model simplex of $K$. By construction, the simplex $\sigma_o$ is obtained from $\sigma$ by a map $g$ which is a contraction of factor $\frac{1}{\ell}$ followed by a translation and we may assume that $\sigma_o \subset \sigma$. Consider a vertex $v_o$ of $\sigma_o$ and denote by $v$ the corresponding vertex in $\sigma$. If $\tau_o$ is the $(q-1)$-face opposite to $v_o$ in $\sigma_o$, a $\frac{\delta}{\ell}$-jiggling $\tau_o'$ of $\tau_o$ corresponds to a $\delta$-jiggling $\tau'$ of the face $\tau$ opposite to $v$ in $\sigma$.\\

For any $\delta$-jiggling $\tau'$ of $\tau$ and any $u \in N_\delta(\st_K(v))$, the sector $S_\varepsilon=S_\varepsilon(\sigma,v,\tau',u)$, satisfies
$$
\frac{m(S_\varepsilon \cap B(v,\delta))}{m(B(v,\delta))} < \frac{1}{A}.
$$
Clearly, a contraction modifies the measures of $B(v,\delta)$ and $S_\varepsilon \cap B(v,\delta)$, but their ratio is unchanged and since $g(B(v,\delta)) = B(v_o,\frac{\delta}{\ell})$, we have
$$
\frac{m\left(g(S_\varepsilon) \cap B(v_o,\frac{\delta}{\ell})\right)}{m\left(B(v_o,\frac{\delta}{\ell})\right)} < \frac{1}{A}.
$$
Now observe that the $q$-plane $L_{\langle x, \tau' \rangle}$ coincides with $L_{\langle g(x),\tau_o'\rangle}$ in $G_q(\R^k)$, so that
$$
g(S_\varepsilon) \cap B\left(v_0,\frac{\delta}{\ell}\right) = \left\{x_o \in B\left(v_o,\frac{\delta}{\ell}\right) : d_\mathcal{H}(L_{\langle x_o,\tau_o' \rangle}, \mathcal{L}^q_u) \leq \varepsilon \right\},
$$
for any $u \in N_\delta(\st_K(\sigma)) \supset N_{\frac{\delta}{\ell}}(\st_{\mathcal{S}_\ell(K)}(\sigma_o))$. Consequently, for any $\frac{\delta}{\ell}$-jiggling $\tau_o'$ of the face $\tau_o$ opposite to $v_o$ in $\sigma_o$, there exists a point $x_o \in B\left(v_o,\frac{\delta}{\ell}\right)$ such that the plane generated by $x_o$ and $\tau_o'$ is at distance at least $\varepsilon$ from $\mathcal{L}^q_{v_o}$. It is remarkable that the choice of $\varepsilon$, made on the model simplices, allows to guarantee that such a property holds for \textsl{any} crystalline subdivision $\mathcal{S}_\ell(K)$. This shows the importance of the particular properties of crystalline subdivisions. \\

With these preliminaries, the proof of Proposition \ref{jiggling} is quite short.

\begin{proof}[Proof of Proposition \ref{jiggling}]
In the sequel, for any $\ell \in \N_0$, any $\delta>0$ and any vertex $w$ of $\mathcal{S}_\ell(K)$, we denote by $N_{\delta/\ell}(w)$ the closed $\frac{\delta}{\ell}$-neighborhood of the star of $w$ with respect to the simplicial complex $\mathcal{S}_\ell(K)$. 

First, fix $\delta >0$ as in \rref{delta} and $\varepsilon>0$ provided by Lemma \ref{epsilon}. Then, choose $\ell$ sufficiently large so that, for any $w \in \mathcal{S}_\ell(K)^{(0)}$ and for any choice of points $u, u'$ in $N_{\delta/\ell}(w)$, the Hausdorff distance between $\mathcal{L}_u^q$ and $\mathcal{L}_{u'}^q$ is less than $\varepsilon/2$. 

Let $w_0, \ldots, w_r$ denote the vertices of $\mathcal{S}_\ell(K)$ and fix $w_0'=w_0$. Assume that for some $p < r$, we have chosen $w'_i \in B\left(w_i, \frac{\delta}{\ell}\right)$ for $i = 0, \ldots, p$ such that for any $q$-simplex $\nu =\langle w_{i_0}, w_{i_1} \ldots, w_{i_q}\rangle$ of $\mathcal{S}_\ell(K)$, with $i_0, ..., i_q \leq p$ and for any $u \in \cup_{i = 0}^q N_{\delta/\ell}(w_i)$,
$$
d_\mathcal{H}(L_{\nu'}, \mathcal{L}_u^q) > \frac{\varepsilon}{2},
$$
where $\nu'$ is the jiggled $q$-simplex $\nu' = \langle w'_{i_0}, \ldots, w'_{i_q} \rangle$.
Consider the vertex $w_{p+1}$. For any $q$-simplex of the type $\nu = \langle w_{p+1}, w_{i_1}, \ldots, w_{i_q} \rangle$, with $i_1 < \ldots < i_q \leq p$, we know that the Lebesgue measure of the set
$$
S'_{\varepsilon}(\nu) =\left\{ x \in B\left(w_{p+1},\frac{\delta}{\ell}\right) : d_\mathcal{H}(L_{\langle x, w'_{i_1}, \ldots, w'_{i_q} \rangle}, \mathcal{L}_{w_{p+1}}^q )\leq \varepsilon \right\}
$$ satisfies 
$$
m\left(S'_{\varepsilon}(\nu)\right) < \frac{m\left(B(w_{p+1},\frac{\delta}{\ell})\right)}{A}.
$$
Now define
$$
S'_{\varepsilon} = \bigcup_\nu S'_{\varepsilon}(\nu),
$$
where $\nu$ runs over all $q$-simplices of $\mathcal{S}_\ell(K)$ of the previously described type. Since the number of simplices in $\st_{\mathcal{S}_\ell(K)}(w_{p+1})$ is bounded by $A$, the previous inequality implies that the Lebesgue measure of $S'_{\varepsilon}$ is smaller than the measure of the ball $B\left(w_{p+1},\frac{\delta}{\ell}\right)$. In particular, there exist points $w_{p+1}'$ in $B\left(w_{p+1},\frac{\delta}{\ell}\right)\backslash S'_{\varepsilon_o}$, which means that for any jiggled $q$-simplex $\nu' = \langle w'_{p+1}, w'_{i_1}, \ldots, w'_{i_q}  \rangle$,
$$
d_\mathcal{H}\left(L_{\nu'},\mathcal{L}_{w_{p+1}}^q\right) > \varepsilon.
$$
Consequently, for each point $u \in N_{\delta/\ell}(w_{p+1})$, we have 
$$
d_\mathcal{H}(L_{\nu'}, \mathcal{L}_u^q) > d_\mathcal{H}(L_{\nu'}, \mathcal{L}_{w_{p+1}}^q) - d_\mathcal{H}(\mathcal{L}_{w_{p+1}}^q,\mathcal{L}_u^q) > \frac{\varepsilon}{2}, 
$$
so that $\Omega_u$ is non-degenerate along $L_{\nu'}$. We may of course repeat this procedure until exhausting all the vertices of $\mathcal S_\ell(K)$, thus achieving that the $(q+1)$-skeleton of $\mathcal S'_\ell(K)$ is in general position with respect to $\Omega$.  
\end{proof}

\section{Construction of approximating PL symplectic manifolds}\label{construction PC}

In this section, we are going to show how to approximate any smooth symplectic form by piecewise constant ones. Our first construction used the jiggling lemma but, as was observed by the referee, the jiggling lemma is in fact not necessary. We nevertheless present here our first construction because we find it slightly more understandable and observe in a remark (\rref{nojiggling}) that general position is superfluous. 

\smallskip

Given a closed symplectic manifold $(M,\omega_M)$ of dimension $2n$, the symplectic jiggling lemma allows us to consider a smooth triangulation $h : |K| \to M$ in general position with respect to  $\omega_M$. On each 2-dimensional simplex $\tau$ of $K$, we first define $\omega_\tau$ as the constant 2-form which admits the same orientation and symplectic area as ${h^*\omega_M}$ on $\tau$,~i.e.
$$
\int_{\tau} \omega_\tau = \int_{\tau} {h^*\omega_M}|_\tau.
$$
(The integrand on the right hand side is really the restriction of $h^*\omega_M$ to $\tau$ of course).

\begin{lemma}\label{PC}
There exists a unique piecewise constant 2-form $\omega$ on $K$ whose restriction to any $2$-face $\tau$ of $K$ is $\omega_\tau$. Moreover, $\omega$ is (simplicially) cohomologous to $h^*\omega_M$.
\end{lemma}

\begin{proof}
Consider a $2n$-simplex $\sigma$ of $K$ and fix one of its vertices $v$. We denote by $v_1, \ldots, v_{2n}$ the other vertices of $\sigma$. 
Let us define $\om$ on $\sigma$ by
$$
\omega(v_i-v,v_j-v)=\omega_\tau(v_i-v,v_j-v),
$$
where $\tau$ is the 2-face spanned by $v$, $v_i$ and $v_j$. We have to show that this form is well-defined, i.e.~that its restriction to a 2-face $\tau_0$ of $\sigma$ which does not contain $v$ coincides with the corresponding $\omega_{\tau_0}$. Set $\nu = <\tau_0, v>$, and observe that, by Stokes' theorem,
$$
\int_{\partial \nu} \omega = \int_{\nu} d\omega = 0.
$$
If $\nu_v^{(2)}$ denotes the set of 2-dimensional faces of $\nu$ which contain $v$, it follows from Stokes'~theorem, applied to $h^*\omega_M$, that 
$$
\int_{\tau_0} \omega = - \sum_{\tau \in \nu^{(2)}_v} \int_{\tau} \omega = - \sum_{\tau \in \nu^{(2)}_v} \int_{\tau} h^*\omega_M = \int_{\tau_0} h^*\omega_M,
$$
where $\tau_0$ and the $\tau$'s are oriented as portions of the boundary of $\nu$.
This shows that $\omega$ is well-defined and does not depend on the choice of $v$. 
Observe that by definition, $\omega$ and $h^*\omega_M$ induce the same class in simplicial cohomology. They are therefore cohomologous.
\end{proof}

The Whitney form constructed in Lemma \ref{PC} is denoted by $\mathcal{C}_K(\omega_M)$. It is symplectic provided it is non-degenerate on each $2n$-simplex of $K$. The following result ensures that it is the case for sufficiently fine triangulations. 

\begin{prop}\label{example PC structure}
Let $(M^{2n},\omega_M)$ be a closed symplectic manifold and $h : |K| \to M$ a Whitehead triangulation in general position with respect to $\omega_M$. Then there exists an $\ell_o$ such that $\mathcal{C}_{\mathcal{S}_\ell(K)}(\omega_M)$ is a PL symplectic structure for all $\ell \geq \ell_o$. 
\end{prop}

\begin{proof} 
Set $\omega_1 = h^*\omega_M$. Let us fix a Euclidean norm $\| \cdot \|$  on the vector bundle $\Lambda^2(T^*M)$. Let $\varepsilon > 0$, we shall prove that there exists an $\ell_o \in \N_0$ such that if $\ell \geq \ell_o$ and $\omega=\mathcal{C}_{\mathcal{S}_\ell(K)}(\omega_M)$, 
we have
$$
\| \omega- \omega_1 \| < \varepsilon,
$$
meaning that $\|\omega_\sigma - (\omega_1)|_\sigma\| < \varepsilon$ for all $\sigma \in \mathcal{S}_\ell(K)$. Since for any $\sigma \in \mathcal{S}_\ell(K)^{(2n)}$, the form $(\omega_1)|_\sigma$ is non-degenerate, this will imply that $\omega_\sigma$ is non-degenerate as well. \\

One of the properties of the crystalline subdivision is that 
$$\max_{\sigma \in \mathcal{S}_\ell(K)} \mathrm{diam}(\sigma)$$ 
tends to 0 when $\ell \to \infty$. We can therefore choose an $\ell_o \geq 1$ such that for all $\ell \geq \ell_o$, any $\sigma \in \mathcal{S}_\ell(K)$ and any $x, y \in \sigma$,
$$
\| (\omega_1)_x - (\omega_1)_y  \| < \frac{\varepsilon}{2}.
$$
Consider a 2-simplex $\tau$ of $\mathcal{S}_\ell(K)$ and observe that the form $\omega_\tau$ necessarily coincides with $(\omega_1)_{x_\tau}|_\tau$ for a certain $x_\tau \in \tau$. This is implied by the fact that $\omega_\tau$ and $\omega_1|_\tau$ have the same volume on $\tau$. Consequently, for each $\tau \in \mathcal{S}_\ell(K)^{(2)}$ and $x \in \tau$, we have
$$
\|(\omega_1)_x|_\tau - \omega_\tau \| < \frac{\varepsilon}{2}.
$$
Finally, fix a point $x$ in a $2n$-simplex $\sigma = <v_0, ..., v_{2n}>$ in $\mathcal{S}_\ell(K)$. If $\tau_{ij} = <v_0, v_i, v_j>$
, we have
$$
\| (\omega_1)_x|_{{\tau}_{ij}} - \omega_{{\tau}_{ij}} \| \leq \| (\omega_1)_x|_{{\tau}_{ij}} - (\omega_1)_{\tilde{x}}|_{\tau_{ij}} \| + \| (\omega_1)_{\tilde{x}}|_{\tau_{ij}} - \omega_{\tau_{ij}} \| < \varepsilon,
$$
for some point $\tilde{x}$ in $\tau_{ij}$.
This shows that $\|(\omega_1)_x - \omega_\sigma \| < \varepsilon$ for any $x \in \sigma$. 
\end{proof}

\begin{rem}\label{nojiggling}
There is in fact no need to require the triangulation to be in general position with respect to $\omega_M$. One could define, for each $2$-simplex $\tau$ of $K$, the $2$-form $\omega_\tau \in \Omega^2(\tau)$ to be the unique constant $2$-form on $\tau$ such that 
$$\int_{(\tau, o)} \omega_\tau = \int_{(\tau, o)} h^*\omega_M,$$
for some (hence any) orientation $o$ of $\tau$. The proofs of \lref{PC} and \pref{example PC structure} do not use the fact that the $\om_\tau$'s are nondegenerate and remain thus valid in this more general setting. \\
\end{rem}

Theorem \ref{example PC structure} provides for a large class of examples of PL symplectic structures. In particular, we have the following.

\begin{cor}
A closed PL manifold which admits a symplectic smoothing supports PL symplectic structures.
\end{cor}

\begin{cor}\label{example PC 2}
Let $(M,\omega_M)$ be a closed symplectic manifold. Then there exists a PL symplectic manifold $(P, \omega_P)$ with a piecewise smooth equivalence $h : P \to M$ such that $\omega_P$ and $h^*\omega_M$ are cohomologous and $C^0$-close.
\end{cor}

Let $K$ be a triangulation of $P$ and consider the PC symplectic form $\omega = \mathcal{C}_K(\omega_M)$ given by Corollary \ref{example PC 2}. We do not know how to construct a PS equivalence between those two symplectic structures but we would like to use what has been done before to prove that the PC volume $(\mathcal{C}_K(\omega_M))^n$ triangulates $\Om_M^n$. Notice first that since $\om$ and $h^*\omega_M$ are cohomologous, they have identical total volume. On the other hand, they do not necessarily share the same volume cocycle. This indicates that if we aim to construct a triangulation of the smooth volume form $\omega_M^n$, we should first change the initial triangulation in order to adjust the volumes of the simplices. \\

\begin{prop}\label{ex PC form}
Let $(M,\omega_M)$ be a closed symplectic manifold. There exists a smooth triangulation $h : |K| \to M$ and a PC symplectic form $\omega$ on $K$, cohomologous and $C^0$-close to $h^*\omega_M$, such that 
$$
\int_{\sigma} \overline{h}^*\omega_M^n = \int_{\sigma} \omega^n
$$
for any $\sigma \in K^{(2n)}$
\end{prop}

\begin{proof}
Consider a triangulation $h_o : |K| \to M$ and a PC symplectic form $\omega$ on $|K|$ obtained from Corollary \ref{example PC 2}. Set $\Omega_M = \om_M^n$ and $\Omega = \om^n$. There exists a smooth function $f : M \to \R_0^+$ identically equals to $1$ near the skeleton of $K$ and such that
$$
\int_{h_o(\sigma)} f\Omega_M = \int_{\sigma} \Omega
$$ 
for $\sigma \in K^{(2n)}$. The volume forms $f\Omega_M$ and $\Omega_M$ have the same total volume, and there exists thus, by Moser's argument, a diffeomorphism $\varphi$ of $M$, supported away from the $2$-skeleton, such that $\varphi^*\Omega_M = f\Omega_M$. The fact that $\varphi$ can be chosen to be supported away from the $2$-skeleton comes from a  the choice of a primitive of $(1-f) \Omega$ that vanishes near $|K^{(2)}|$ (choose a first primitive $\beta_o$ and then observe that near $|K^{(2)}|$, it is necessarily closed and therefore exact; if $\gamma$ is a primitive of $\beta_o$ near $|K^{(2)}|$, define $\beta = \beta_o - d(\eta \gamma)$, for a bump function $\eta$ supported near $|K^{(2)}|$ and identically equals to $1$ near $|K^{(2)}|$).
Now define a new triangulation $h = \varphi \circ h_o : |K| \to M$. By construction, it satisfies 
$$
h^*\Omega_M = h_o^*(\varphi^*\Omega_M)=h_o^*(f\Omega_M),
$$
which implies that for any $2n$-dimensional simplex $\sigma$,
$$
\int_{\sigma} h^*\Omega_M = \int_{h_o(\sigma)} f \Omega_M = \int_{\sigma} \Omega.
$$
Moreover, the PC symplectic form associated to $h^*\omega_M$ is $\omega$ as well, since $h|_{K^{(2)}} = h_o|_{K^{(2)}}$. In particular, $h^*\omega_M$ is cohomologous to $h_o^*\omega_M$ and thus to $\omega$.
Observe finally that if the triangulation is sufficiently fine, the diffeomorphism $\varphi$ can be chosen close to the identity~: this follows from the fact that since the volume forms $h^*\Omega_M$ and $\Omega$ are close, we can consider a primitive of their difference which is close to zero. This implies that $\omega$ is also $C^0$-close to $h^*\omega_M$.
\end{proof}

\begin{cor}
Let $(M,\omega_M)$ be a closed symplectic manifold. There exists a smooth triangulation $\tilde{h} : |K| \to M$ and a PC symplectic form $\omega$ cohomologuous and $C^0$-close to $\tilde{h}^*\om_M$ on $K$ such that $\tilde{h}^*\Omega_M=\Omega$.
\end{cor}

\begin{proof}
Start with the triangulation $h : |K| \to M$ provided by Proposition \ref{ex PC form}. Since $h^*\Omega_M$ and $\Omega$ have the same volumes on the $2n$-simplices of $K$, it follows from Theorem~\ref{existence} that there exists a subdivision $\tilde{K}$ of $K$ and a triangulation $\tilde{h} : |\tilde{K}| \to M$ such that $\tilde{h}^*\Omega_M = \Omega$.
\end{proof}

\bibliography{bibjl}
\end{document}

%% file: subdivision2.pdf_tex
\begingroup%
  \makeatletter%
  \providecommand\color[2][]{%
    \errmessage{(Inkscape) Color is used for the text in Inkscape, but the package 'color.sty' is not loaded}%
    \renewcommand\color[2][]{}%
  }%
  \providecommand\transparent[1]{%
    \errmessage{(Inkscape) Transparency is used (non-zero) for the text in Inkscape, but the package 'transparent.sty' is not loaded}%
    \renewcommand\transparent[1]{}%
  }%
  \providecommand\rotatebox[2]{#2}%
  \ifx\svgwidth\undefined%
    \setlength{\unitlength}{1069.30662715bp}%
    \ifx\svgscale\undefined%
      \relax%
    \else%
      \setlength{\unitlength}{\unitlength * \real{\svgscale}}%
    \fi%
  \else%
    \setlength{\unitlength}{\svgwidth}%
  \fi%
  \global\let\svgwidth\undefined%
  \global\let\svgscale\undefined%
  \makeatother%
  \begin{picture}(1,0.44834105)%
    \put(0,0){\includegraphics[width=\unitlength,page=1]{subdivision2.pdf}}%
    \put(0.33466475,0.45168001){\color[rgb]{0,0,0}\makebox(0,0)[lt]{\begin{minipage}{0.08416669\unitlength}\raggedright $v_3$\end{minipage}}}%
    \put(0.92483356,0.29737439){\color[rgb]{0,0,0}\makebox(0,0)[lt]{\begin{minipage}{0.08416669\unitlength}\raggedright $v_2$\end{minipage}}}%
    \put(0.89677799,0.04487432){\color[rgb]{0,0,0}\makebox(0,0)[lt]{\begin{minipage}{0.09819447\unitlength}\raggedright $v_1$\end{minipage}}}%
    \put(-0.00100003,0.03084654){\color[rgb]{0,0,0}\makebox(0,0)[lt]{\begin{minipage}{0.11222225\unitlength}\raggedright $v_0$\end{minipage}}}%
    \put(0.39177786,0.24126326){\color[rgb]{0,0,0}\makebox(0,0)[lt]{\begin{minipage}{0.05611113\unitlength}\raggedright $b$\end{minipage}}}%
    \put(0,0){\includegraphics[width=\unitlength,page=2]{subdivision2.pdf}}%
  \end{picture}%
\endgroup%

%% file: interpolation3.pdf_tex
\begingroup%
  \makeatletter%
  \providecommand\color[2][]{%
    \errmessage{(Inkscape) Color is used for the text in Inkscape, but the package 'color.sty' is not loaded}%
    \renewcommand\color[2][]{}%
  }%
  \providecommand\transparent[1]{%
    \errmessage{(Inkscape) Transparency is used (non-zero) for the text in Inkscape, but the package 'transparent.sty' is not loaded}%
    \renewcommand\transparent[1]{}%
  }%
  \providecommand\rotatebox[2]{#2}%
  \ifx\svgwidth\undefined%
    \setlength{\unitlength}{1149.7711291bp}%
    \ifx\svgscale\undefined%
      \relax%
    \else%
      \setlength{\unitlength}{\unitlength * \real{\svgscale}}%
    \fi%
  \else%
    \setlength{\unitlength}{\svgwidth}%
  \fi%
  \global\let\svgwidth\undefined%
  \global\let\svgscale\undefined%
  \makeatother%
  \begin{picture}(1,0.76051029)%
    \put(0,0){\includegraphics[width=\unitlength,page=1]{interpolation3.pdf}}%
    \put(0.07978343,0.0286878){\color[rgb]{0,0,0}\makebox(0,0)[lt]{\begin{minipage}{0.19569112\unitlength}\raggedright $\varepsilon$ \end{minipage}}}%
    \put(0.65512856,0.0286878){\color[rgb]{0,0,0}\makebox(0,0)[lt]{\begin{minipage}{0.08479948\unitlength}\raggedright $R$\end{minipage}}}%
    \put(0.95914184,0.06841873){\color[rgb]{0,0,0}\makebox(0,0)[lt]{\begin{minipage}{0.07827645\unitlength}\raggedright $x$\end{minipage}}}%
    \put(-0.00093004,0.76361558){\color[rgb]{0,0,0}\makebox(0,0)[lt]{\begin{minipage}{0.06523037\unitlength}\raggedright $y$\end{minipage}}}%
    \put(0,0){\includegraphics[width=\unitlength,page=2]{interpolation3.pdf}}%
    \put(0.66059865,0.6533862){\color[rgb]{0,0,0}\makebox(0,0)[lt]{\begin{minipage}{0.16959897\unitlength}\raggedright $\mathrm{Id}$\end{minipage}}}%
    \put(0.20332712,0.42237797){\color[rgb]{0,0,0}\makebox(0,0)[lt]{\begin{minipage}{0.15002986\unitlength}\raggedright $\varphi(x)$\end{minipage}}}%
  \end{picture}%
\endgroup%

%% file: interpolation2.pdf_tex
\begingroup%
  \makeatletter%
  \providecommand\color[2][]{%
    \errmessage{(Inkscape) Color is used for the text in Inkscape, but the package 'color.sty' is not loaded}%
    \renewcommand\color[2][]{}%
  }%
  \providecommand\transparent[1]{%
    \errmessage{(Inkscape) Transparency is used (non-zero) for the text in Inkscape, but the package 'transparent.sty' is not loaded}%
    \renewcommand\transparent[1]{}%
  }%
  \providecommand\rotatebox[2]{#2}%
  \ifx\svgwidth\undefined%
    \setlength{\unitlength}{1149.7711291bp}%
    \ifx\svgscale\undefined%
      \relax%
    \else%
      \setlength{\unitlength}{\unitlength * \real{\svgscale}}%
    \fi%
  \else%
    \setlength{\unitlength}{\svgwidth}%
  \fi%
  \global\let\svgwidth\undefined%
  \global\let\svgscale\undefined%
  \makeatother%
  \begin{picture}(1,0.76051029)%
    \put(0,0){\includegraphics[width=\unitlength,page=1]{interpolation2.pdf}}%
    \put(0.07978343,0.0286878){\color[rgb]{0,0,0}\makebox(0,0)[lt]{\begin{minipage}{0.19569112\unitlength}\raggedright $\varepsilon$ \end{minipage}}}%
    \put(0.65512856,0.0286878){\color[rgb]{0,0,0}\makebox(0,0)[lt]{\begin{minipage}{0.08479948\unitlength}\raggedright $R$\end{minipage}}}%
    \put(0.95914184,0.06841873){\color[rgb]{0,0,0}\makebox(0,0)[lt]{\begin{minipage}{0.07827645\unitlength}\raggedright $x$\end{minipage}}}%
    \put(0.69169731,0.73376855){\color[rgb]{0,0,0}\makebox(0,0)[lt]{\begin{minipage}{0.15655289\unitlength}\raggedright $h(x)$\end{minipage}}}%
    \put(-0.00093004,0.76361558){\color[rgb]{0,0,0}\makebox(0,0)[lt]{\begin{minipage}{0.06523037\unitlength}\raggedright $y$\end{minipage}}}%
    \put(0,0){\includegraphics[width=\unitlength,page=2]{interpolation2.pdf}}%
  \end{picture}%
\endgroup%

%% file: sub.pdf_tex
\begingroup%
  \makeatletter%
  \providecommand\color[2][]{%
    \errmessage{(Inkscape) Color is used for the text in Inkscape, but the package 'color.sty' is not loaded}%
    \renewcommand\color[2][]{}%
  }%
  \providecommand\transparent[1]{%
    \errmessage{(Inkscape) Transparency is used (non-zero) for the text in Inkscape, but the package 'transparent.sty' is not loaded}%
    \renewcommand\transparent[1]{}%
  }%
  \providecommand\rotatebox[2]{#2}%
  \ifx\svgwidth\undefined%
    \setlength{\unitlength}{864.75670126bp}%
    \ifx\svgscale\undefined%
      \relax%
    \else%
      \setlength{\unitlength}{\unitlength * \real{\svgscale}}%
    \fi%
  \else%
    \setlength{\unitlength}{\svgwidth}%
  \fi%
  \global\let\svgwidth\undefined%
  \global\let\svgscale\undefined%
  \makeatother%
  \begin{picture}(1,0.88327176)%
    \put(0,0){\includegraphics[width=\unitlength,page=1]{sub.pdf}}%
    \put(0.8202166,0.75467855){\color[rgb]{0,0,0}\makebox(0,0)[lt]{\begin{minipage}{0.12142143\unitlength}\raggedright $f(v_{i_0})$\end{minipage}}}%
    \put(-0.00123657,0.7435276){\color[rgb]{0,0,0}\makebox(0,0)[lt]{\begin{minipage}{0.18832709\unitlength}\raggedright $f(v_{i_1})$\end{minipage}}}%
    \put(0.01734834,0.07818767){\color[rgb]{0,0,0}\makebox(0,0)[lt]{\begin{minipage}{0.17345919\unitlength}\raggedright $f(v_{i_2})$\end{minipage}}}%
    \put(0,0){\includegraphics[width=\unitlength,page=2]{sub.pdf}}%
  \end{picture}%
\endgroup%

%% file: Jiggling+volume-revision-JSG.bbl
\begin{thebibliography}{BMPR18}

\bibitem[BMPR18]{bruveris}
M.~Bruveris, P.W. Michor, A.~Parusi\'{n}ski, and A.~Rainer.
\newblock Moser's theorem on manifolds with corners.
\newblock {\em Proc. Amer. Math. Soc.}, 146(11):4889--4897, 2018.

\bibitem[Dis19]{D19}
Julie Distexhe.
\newblock {\em Triangulating symplectic manifolds}.
\newblock PhD thesis, Universit\'{e} libre de Bruxelles, 2019.

\bibitem[Eli]{YE}
Yakov Eliashberg.
\newblock Private communication.

\bibitem[Gra98]{gratza}
B.~Gratza.
\newblock Piecewise linear approximations in symplectic geometry.
\newblock 1998.
\newblock [PhD Thesis (ETH Zurich)].

\bibitem[JRT18]{rollin}
F.~Jauberteau, Y.~Rollin, and S.~Tapie.
\newblock Discrete geometry and isotropic surfaces.
\newblock 2018.
\newblock [arXiv:1802.08712v2].

\bibitem[Lur09]{Lurie}
J.~Lurie.
\newblock Course syllabus for 18.937: Topics in geometric topology.
\newblock 2009.
\newblock [http://www.math.harvard.edu/ lurie/].

\bibitem[Pan09]{panov}
D.~Panov.
\newblock Polyhedral {K}\"{a}hler manifolds.
\newblock {\em Geom. Topol.}, 13(4):2205--2252, 2009.

\bibitem[Rol20]{rollin2020polyhedral}
Yann Rollin.
\newblock Polyhedral approximation by lagrangian and isotropic tori.
\newblock {\em arXiv preprint arXiv:2012.05777}, 2020.

\bibitem[Rol21]{rollin2021polyhedral}
Yann Rollin.
\newblock Polyhedral symplectic maps and hyperk\"ahler geometry.
\newblock {\em arXiv preprint arXiv:2110.02679}, 2021.

\bibitem[Sul77]{sullivan}
D.~Sullivan.
\newblock Infinitesimal computations in topology.
\newblock {\em Inst. Hautes \'{E}tudes Sci. Publ. Math.}, (47):269--331 (1978),
  1977.

\bibitem[Thu75]{thurstonjig}
W.P. Thurston.
\newblock The theory of foliations of codimension greater than one.
\newblock {\em Differential geometry ({P}roc. {S}ympos. {P}ure {M}ath., {V}ol.
  {XXVII}, {S}tanford {U}niv., {S}tanford, {C}alif., 1973), {P}art 1}, page
  321, 1975.

\bibitem[Whi12]{git}
H.~Whitney.
\newblock {\em Geometric Integration Theory}.
\newblock Princeton mathematical series. Dover Publications, 2012.

\end{thebibliography}
